\documentclass[12pt]{amsart}

\usepackage{amscd,latexsym,amsthm,amsfonts,amssymb,amsmath,amsxtra}
\usepackage[mathscr]{eucal}
\usepackage{hyperref}
\renewcommand\theequation{\thesection.\arabic{equation}}

\newcommand{\BA}{{\mathbb {A}}}

\newcommand{\BC}{{\mathbb {C}}}

\newcommand{\BR}{{\mathbb {R}}}

\newcommand{\BZ}{{\mathbb {Z}}}

\newcommand{\CA}{{\mathcal {A}}}

\newcommand{\CH}{{\mathcal {H}}}

\newcommand{\CM}{{\mathcal {M}}}

\newcommand{\CO}{{\mathcal {O}}}

\newcommand{\bs}{\backslash}

\newcommand{\cusp}{{\mathrm{cusp}}}

\newcommand{\disc}{{\mathrm{disc}}}

\newcommand{\GL}{{\mathrm{GL}}}

\newcommand{\Ind}{{\mathrm{Ind}}}

\newcommand{\Jord}{{\mathrm{Jord}}}

\newcommand{\SO}{{\mathrm{SO}}}

\newcommand{\Sym}{{\mathrm{Sym}}}

\newcommand{\Sp}{{\mathrm{Sp}}}

\newcommand{\wt}{\widetilde}

\newcommand{\ol}{\overline}
\newcommand{\ul}{\underline}

\def\bks{{\backslash}}

\def\diag{{\rm diag}}

\newtheorem{thm}{Theorem}[section]
\newtheorem{cor}[thm]{Corollary}
\newtheorem{lem}[thm]{Lemma}
\newtheorem{prop}[thm]{Proposition}
\newtheorem {conj}[thm]{Conjecture}

\newtheorem {ques/conj}[thm]{Question/Conjecture}

\newtheorem{defn}[thm]{Definition}
\newtheorem{rmk}[thm]{Remark}

\makeatletter

\newcommand{\Rmnum}[1]{\expandafter\@slowromancap\romannumeral #1@}
\makeatother

\begin{document}
\renewcommand{\theequation}{\arabic{equation}}
\numberwithin{equation}{section}

\title[Arthur Parameters and Fourier Coefficients]{Arthur Parameters and Fourier coefficients for Automorphic Forms on Symplectic Groups}

\author{Dihua Jiang}
\address{School of Mathematics\\
University of Minnesota\\
Minneapolis, MN 55455, USA}
\email{dhjiang@math.umn.edu}

\author{Baiying Liu}
\address{Department of Mathematics\\
University of Utah\\
Salt Lake City, UT 84112, USA}
\email{liu@math.utah.edu}

\subjclass[2000]{Primary 11F70, 22E50; Secondary 11F85, 22E55}

\date{\today}


\keywords{Arthur Parameters, Fourier Coefficients, Unipotent Orbits, Automorphic Forms}

\thanks{The research of the first named author is supported in part by the NSF Grants DMS--1301567, and the research of the second
named author is supported in part by NSF Grants DMS-1302122, and in part by a postdoc research fund from Department of Mathematics, University of Utah}

\begin{abstract}
We study the structures of Fourier coefficients of automorphic forms on symplectic groups
based on their local and global structures related to Arthur parameters. This is a first step towards the general conjecture
on the relation between the structure of Fourier coefficients and Arthur parameters for automorphic forms occurring in the discrete spectrum,
given by the first named author in \cite{J14}.
\end{abstract}

\maketitle


\section{Introduction}

In the classical theory of automorphic forms, Fourier coefficients encode abundant arithmetic information of automorphic forms on one
hand. On the other hand, Fourier coefficients bridges the connection from harmonic analysis to number theory via automorphic forms.
In the modern theory of automorphic forms, i.e. the theory of automorphic representations of reductive algebraic groups defined over
a number field $F$ (or a global field), Fourier coefficients continue to play the indispensable role in the last half century.

In the theory of automorphic forms on $\GL_n$, the Whittaker-Fourier coefficients played a fundamental role due to the fact that
every cuspidal automorphic representation of $\GL_n(\BA)$, where $\BA$ is the ring of adeles of $F$, has a non-zero Whittaker-Fourier
coefficient, a classical theorem of Piatetski-Shapiro and Shalika (\cite{PS79} and \cite{S74}).
This result has been extended to the discrete spectrum of $\GL_n(\BA)$ in \cite{JL13}. In general,
due to the nature of the discrete spectrum of square-integrable automorphic forms on reductive algebraic groups $G$, one has
to consider more general version of Fourier coefficients, i.e. Fourier coefficients of automorphic forms attached to unipotent orbits on $G$.
Such general Fourier coefficients of automorphic forms, including Bessel-Fourier coefficients and Fourier-Jacobi coefficients have been
widely used in theory of automorphic $L$-functions via integral representation method (see \cite{GJRS11} and \cite{JZ14}, for instance),
in automorphic descent method of Ginzburg, Rallis and Soudry to produce special cases of explicit Langlands functorial transfers (\cite{GRS11}),
and in the Gan-Gross-Prasad conjecture on vanishing of the central value of certain automorphic $L$-functions of symplectic type
(\cite{GJR04} and \cite{GGP12}). More recent applications of such general Fourier coefficients to explicit constructions of endoscopy
transfers for classical groups can be found in \cite{J14} (and also in \cite{G12} for split classical groups).


We recall from \cite{JL15a} the definition of Fourier coefficients of automorphic forms attached to unipotent orbits.
Take $G_n=\Sp_{2n}$ to be the symplectic group with a Borel subgroup $B=TU$, where the
maximal torus $T$ consists of all diagonal matrices of form:
$$
\diag(t_1,\cdots,t_n;t_n^{-1},\cdots,t_1^{-1})
$$
and the unipotent radical of $B$ consists of all upper unipotent matrices in $\Sp_{2n}$. This choice fixes a root datum of $\Sp_{2n}$.

Let $\ol{F}$ be the algebraic closure of the number field $F$. The set of all unipotent adjoint orbits of $G_n(\ol{F})$ is parameterized by
the set of partitions of $2n$ whose odd parts occur with even multiplicity (see \cite{CM93}, \cite{N11} and \cite{W01}, for instance).
We may call them symplectic partitions of $2n$. When we consider $G_n$ over $F$, the symplectic partitions of $2n$ parameterize
the $F$-stable unipotent orbits of $G_n(F)$.

As in \cite[Section 2]{JL15a}, for each symplectic partition $\ul{p}$ of $2n$, or equivalently each $F$-stable unipotent orbit $\CO_{\ul{p}}$,
via the standard $\mathfrak{sl}_2(F)$-triple, one may construct an $F$-unipotent subgroup $V_{\ul{p},2}$. In this case, the $F$-rational unipotent
orbits in the $F$-stable unipotent orbit $\CO_{\ul{p}}$ are parameterized by a datum $\ul{a}$ (see \cite[Section 2]{JL15a} for detail).
This datum defines a character $\psi_{\ul{p},\ul{a}}$ of $V_{\ul{p},2}(\BA)$, which is trivial on $V_{\ul{p},2}(F)$.

For an arbitrary automorphic form $\varphi$ on $G_n(\BA)$, the $\psi_{\underline{p}, \underline{a}}$-Fourier coefficient
of $\varphi$ is defined by
\begin{equation}\label{fc}
\varphi^{\psi_{\underline{p}, \underline{a}}}(g):=\int_{V_{\underline{p}, 2}(F)\bks V_{\underline{p}, 2}(\BA)}
\varphi(vg) \psi^{-1}_{\underline{p}, \underline{a}}(v) dv.
\end{equation}
When an irreducible automorphic representation $\pi$ of $G_n(\BA)$ is generated by automorphic forms $\varphi$,
we say that $\pi$ has a nonzero $\psi_{\underline{p}, \underline{a}}$-Fourier coefficient or
a nonzero Fourier coefficient attached to a (symplectic) partition $\ul{p}$ if there exists an
automorphic form $\varphi$ in the space of $\pi$ with a nonzero $\psi_{\underline{p}, \underline{a}}$-Fourier coefficient
$\varphi^{\psi_{\underline{p}, \underline{a}}}(g)$, for some choice of $\ul{a}$.

For any irreducible automorphic representation $\pi$
of $G_n(\BA)$, as in \cite{J14}, we define $\frak{p}^m(\pi)$ (which corresponds to $\mathfrak{n}^m(\pi)$ in the notation of \cite{J14})
to be the set of all symplectic partitions $\underline{p}$ which have the properties
that $\pi$ has a nonzero $\psi_{\underline{p}, \underline{a}}$-Fourier coefficient
for some choice of $\underline{a}$, and for any ${\underline{p}'} > \underline{p}$ (with
the natural ordering of partitions), $\pi$ has no nonzero Fourier coefficients
attached to ${\underline{p}'}$.

It is an interesting problem to determine the structure of the set $\frak{p}^m(\pi)$ for any given irreducible automorphic
representation $\pi$ of $G_n(\BA)$. When $\pi$ occurs in the discrete spectrum of square integrable automorphic functions on
$G_n(\BA)$, the global Arthur parameter attached to $\pi$ (\cite{Ar13}) is clearly a fundamental invariant for $\pi$.
We are going to recall a conjecture made in \cite{J14} which relates the structure of the global
Arthur parameter of $\pi$ to the structure of the set $\frak{p}^m(\pi)$. To do so, we briefly recall the endoscopic classification
of the discrete spectrum for $G_n(\BA)$ from \cite{Ar13}.

The set of global Arthur parameters for the discrete spectrum of $G_n=\Sp_{2n}$ is denoted, as in \cite{Ar13},
by $\wt{\Psi}_2(\Sp_{2n})$, the elements of which are of the form
\begin{equation}\label{psin}
\psi:=\psi_1\boxplus\psi_2\boxplus\cdots\boxplus\psi_r,
\end{equation}
where $\psi_i$ are pairwise different simple global Arthur parameters of orthogonal type and have the form
$\psi_i=(\tau_i,b_i)$. Here $\tau_i\in\CA_\cusp(\GL_{a_i})$,
(the set of equivalence classes of irreducible cuspidal automorphic representations of $\GL_{a_i}(\BA)$),
 $2n+1 = \sum_{i=1}^r a_ib_i$ (since the dual group of $\Sp_{2n}$ is $\SO_{2n+1}(\BC)$),
and $\prod_i \omega_{\tau_i}^{b_i} = 1$ (the condition
on the central characters of the parameter $\psi$), following \cite[Section 1.4]{Ar13}. More precisely, for each $1 \leq i \leq r$, $\psi_i=(\tau_i,b_i)$
satisfies the following conditions:
if $\tau_i$ is of symplectic type (i.e., $L(s, \tau_i, \wedge^2)$ has a pole at $s=1$), then $b_i$ is even; if $\tau_i$ is of orthogonal type (i.e., $L(s, \tau_i, \Sym^2)$ has a pole at $s=1$), then $b_i$ is odd. Given a global Arthur parameter $\psi$ as above, recall from \cite{J14} that $\ul{p}(\psi)=[(b_1)^{a_1} \cdots (b_r)^{a_r}]$ is the partition attached to $(\psi, G^{\vee}(\BC))$.

\begin{thm}[Theorem 1.5.2, \cite{Ar13}] For each global Arthur parameter $\psi\in\wt{\Psi}_2(\Sp_{2n})$ a global Arthur
packet $\wt{\Pi}_\psi$ is defined. The discrete spectrum of $\Sp_{2n}(\BA)$ has the following decomposition
$$
L^2_\disc(\Sp_{2n}(F)\bks\Sp_{2n}(\BA))
\cong\oplus_{\psi\in\wt{\Psi}_2(\Sp_{2n})}\oplus_{\pi\in\wt{\Pi}_\psi(\epsilon_\psi)}\pi,
$$
where $\wt{\Pi}_\psi(\epsilon_\psi)$ denotes the subset of $\wt{\Pi}_\psi$ consisting of members which occur in the
discrete spectrum.
\end{thm}

As in \cite{J14}, one may call $\wt{\Pi}_\psi(\epsilon_\psi)$ the automorphic $L^2$-packet attached to $\psi$.
For $\pi\in\wt{\Pi}_\psi(\epsilon_\psi)$, the structure of the global Arthur parameter $\psi$ deduces constraints on
the structure of $\frak{p}^m(\pi)$, which is given by the following conjecture.

\begin{conj}[Conjecture 4.2, \cite{J14}]\label{cubmfc}
For any $\psi\in\wt{\Psi}_2(\Sp_{2n})$, let $\wt{\Pi}_{\psi}(\epsilon_\psi)$ be the automorphic $L^2$-packet attached to $\psi$.
Assume that $\underline{p}(\psi)$ is the partition attached to $(\psi,G^\vee(\BC))$. Then the following hold.
\begin{enumerate}
\item[(1)] Any symplectic partition $\ul{p}$ of $2n$, if $\ul{p}>\eta_{{\frak{g}^\vee,\frak{g}}}(\underline{p}(\psi))$, does
not belong to $\frak{p}^m(\pi)$ for any $\pi\in\wt{\Pi}_{\psi}(\epsilon_\psi)$.
\item[(2)] For a $\pi\in\wt{\Pi}_{\psi}(\epsilon_\psi)$, any partition $\ul{p}\in\frak{p}^m(\pi)$ has the property that
$\ul{p}\leq \eta_{{\frak{g}^\vee,\frak{g}}}(\ul{p}(\psi))$.
\item[(3)] There exists at least one member $\pi\in\wt{\Pi}_{\psi}(\epsilon_\psi)$ having the property that
$\eta_{{\frak{g}^\vee,\frak{g}}}(\ul{p}(\psi))\in \frak{p}^m(\pi)$.
\end{enumerate}
Here $\eta_{{\frak{g}^\vee,\frak{g}}}$ denotes the Barbasch-Vogan duality map from the partitions for the dual group $G^\vee(\BC)$ to
the partitions for $G$.
\end{conj}

We refer to \cite[Section 4]{J14} for more discussion on this conjecture and related topics. We note that
the natural ordering of partitions is a partial ordering, and Part (2) of Conjecture \ref{cubmfc} is to rule out partitions which
are not related to the partition $\eta_{{\frak{g}^\vee,\frak{g}}}(\ul{p}(\psi))$. One may combine Parts (1) and (2) of Conjecture \ref{cubmfc}
into one statement. However, due to the technical reasons, it may be better to separate Part (1) from Part (2).

This paper is part of our on-going project to confirm Conjecture \ref{cubmfc} and is to prove

\begin{thm}\label{main}
Part (1) of Conjecture \ref{cubmfc} holds for any $\psi\in\wt{\Psi}_2(\Sp_{2n})$.
\end{thm}

The proof of Theorem \ref{main} takes steps which combine local and global arguments.
Some discussions on Part (2) of Conjecture \ref{cubmfc} will be given in Section 6.3.
We expect that the refinement of these arguments will
be able to prove Part (2) of Conjecture \ref{cubmfc} in general.
This will be considered in our forthcoming work.
Of course, Part (3) of Conjecture \ref{cubmfc} is global in nature and will be considered
by extending the arguments in \cite{JL15a}.

In \cite{Liu13}, based on the results in \cite{JLZ13} on construction of residual representations,
the second named author confirmed Part (3) of Conjecture \ref{cubmfc} for the following family of special non-generic
global Arthur parameters for $\Sp_{4mn}$: $\psi=(\tau, 2m) \boxplus (1_{\GL_1(\BA)}, 1)$, where $\tau$ is an irreducible cuspidal automorphic representation of
$\GL_{2n}(\BA)$, with the properties that
$L(s, \tau, \wedge^2)$ has a simple
pole at $s=1$, and
$L(\frac{1}{2}, \tau) \neq 0$.
Note that by Theorem \ref{main}, Proposition \ref{prop3} and Remark \ref{rmk8}, the first two parts of Conjecture \ref{cubmfc} hold for these
global Arthur parameters. Therefore, Conjecture \ref{cubmfc} is confirmed for this
family of global Arthur parameters. Of course, as discussed in \cite[Section 4]{J14}, when the global Arthur parameter $\psi$ is generic,
Conjecture \ref{cubmfc} can be viewed as the global version of the Shahidi conjecture, which is now a consequence of \cite{Ar13}
and \cite{GRS11}. We refer to \cite[Section 3]{JL15b} for detailed discussion of this issue and related problems. 

In order to prove Theorem \ref{main}, we first consider the unramified local component $\pi_v$ of an irreducible unitary automorphic
representation $\pi$ of $\Sp_{2n}(\BA)$ at one finite local place $v$ of $F$. The structure of unramified unitary dual of $\Sp_{2n}(F_v)$
was determined by D. Barbasch in \cite{Bar10} and by G. Muic and M. Tadic in \cite{MT11} with different approaches. We recall
from \cite{MT11} the results on unramified unitary dual and determine, for any given global Arthur parameter $\psi\in\wt{\Psi}_2(\Sp_{2n})$,
the unramified components $\pi_v$ of any $\pi\in\wt{\Pi}_{\psi}(\epsilon_\psi)$ in terms of the classification data in \cite{MT11}.
The Fourier coefficients for $\pi$ produce the corresponding twisted Jacquet modules for $\pi_v$. In Section 3, we show in Lemmas
\ref{lemv1} and \ref{lemv2} the vanishing of certain twisted Jacquet modules for the unramified unitary representations $\pi_v$,
which builds up first local constraints for the vanishing of Fourier coefficients of $\pi$. In Section 4, based on the local results
in Sections 2 and 3, we come back to the global situation and prove vanishing of certain Fourier coefficients of $\pi$. Here we use
global techniques developed through the work of \cite{GRS03}, \cite{GRS11}, and \cite{JL15a}, in particular, the results on Fourier
coefficients associated to composites of partitions. The main results in Section 4 are Theorems \ref{thmub2} and \ref{thmub3}, which
establish the vanishing of Fourier coefficients of $\pi$ whose unramified local component $\pi_v$ is strongly negative.
The general case is done in Section 5 (Theorems \ref{thmub1} and \ref{thmub4}). In the last section (Section 6), we first prove in Propositions \ref{propsq} and
\ref{propap}. They imply that
for a given global Arthur parameter $\psi$, there are infinitely many unramified, finite local places $v$ of $F$, where the unramified
local components $\tau_{i,v}$ have trivial central characters.
With such refined results on the central characters of $\tau_{i,v}$, we are able to finish the proof of Theorem \ref{main}
in Section 6.2, by combining
all the results established in the previous sections.

We would like to thank the referee for the careful reading of the paper and helpful comments and suggestions.



\section{Unramified Unitary Dual and Arthur Parameters}

In this section, we recall the classification of the unramified unitary dual of $p$-adic symplectic groups, which was obtained by
Barbasch in \cite{Bar10} and by Muic and Tadic in \cite{MT11}. In terms of the structure of the unramified unitary dual of $p$-adic
$\Sp_{2n}$, we try to understand the structure of the unramified local component $\pi_v$ of an irreducible automorphic
representation $\pi=\otimes_v\pi_v$ of $\Sp_{2n}(\BA)$ belonging to an automorphic $L^2$-packet $\wt{\Pi}_{\psi}(\varepsilon_{\psi})$
for an arbitrary global Arthur parameter $\psi \in \wt{\Psi}_2(\Sp_{2n})$.


\subsection{Unramified Unitary Dual of Symplectic Groups}

The unramified unitary dual of split classical groups was classified by
Barbasch in \cite{Bar10} (both real and $p$-adic cases), and by
Muic and Tadic in \cite{MT11} ($p$-adic case), using different methods.
We follow the approach in \cite{MT11} for $p$-adic symplectic groups.

Let $v$ be a finite local place of the given number field $F$.
The classification in \cite{MT11} starts from classifying two special families of irreducible unramified
representations of $\Sp_{2n}(F_v)$ that are called {\bf strongly negative} and {\bf negative},
respectively. We refer to \cite{M06} for definitions of strongly negative and negative
representations, respectively, and for more related discussion on those two families of unramified representations.
In the following, we recall from \cite{MT11} the classification of these two families in terms of {\bf Jordan blocks}, which
also provide explicit construction of the two families of unramified representations.

A pair $(\chi, m)$, where $\chi$ is an unramified unitary
character of $F_v^*$ and $m \in \BZ_{>0}$, is called
a {\bf Jordan block}.
Define $\rm{Jord}_{sn}(n)$ to be the collection of all sets
$\rm{Jord}$ of the following form:
\begin{equation}\label{sec7equ1}
\{(\lambda_0, 2n_1+1), \ldots, (\lambda_0, 2n_k+1), (1_{\GL_1}, 2m_1+1),
\ldots, (1_{\GL_1}, 2m_l+1)\}
\end{equation}
where $\lambda_0$ is the unique non-trivial unramified
unitary character of $F_v^*$ of order 2,
given by the local Hilbert symbol $(\delta, \cdot)_{F_v^*}$, with $\delta$ being a non-square unit in $\CO_{F_v}$;
$k$ is even,
$$
0 \leq n_1 < n_2 < \cdots < n_k, \ \ \ \ 0 \leq m_1 < m_2 < \cdots < m_l;
$$
and 
$$
\sum_{i=1}^k (2n_i+1) + \sum_{j=1}^l (2m_j+1) = 2n+1.
$$
It is easy to see that $l$ is automatically odd.

For each $\rm{Jord} \in \rm{Jord}_{sn}(n)$, we can associate a
representation $\sigma(\rm{Jord})$, which is the unique irreducible
unramified subquotient of the following induced
representation
\begin{align}\label{sec7equ2}
\begin{split}
& \nu^{\frac{n_{k-1}-n_k}{2}} \lambda_0({\det}_{n_{k-1}+n_k+1})
\times
\nu^{\frac{n_{k-3}-n_{k-2}}{2}} \lambda_0({\det}_{n_{k-3}+n_{k-2}+1})\\
& \times \cdots \times
\nu^{\frac{n_{1}-n_2}{2}} \lambda_0({\det}_{n_{1}+n_2+1})\\
& \times
\nu^{\frac{m_{l-1}-m_l}{2}} 1_{{\det}_{m_{l-1}+m_l+1}}
\times
\nu^{\frac{m_{l-3}-m_{l-2}}{2}} 1_{{\det}_{m_{l-3}+m_{l-2}+1}}\\
& \times \cdots \times
\nu^{\frac{m_{2}-m_3}{2}} 1_{{\det}_{m_{2}+m_3+1}}
\rtimes 1_{\Sp_{2m_1}}.
\end{split}
\end{align}

\begin{thm}[Theorem 5-8, \cite{MT11}] \label{thm8}
Assume that $n > 0$. The map $\rm{Jord} \mapsto \sigma(\rm{Jord})$ defines a
one-to-one correspondence between the set
$\rm{Jord}_{sn}(n)$ to the set of all
irreducible strongly negative
unramified representations of $\Sp_{2n}(F_v)$.
\end{thm}

Note that $1_{\Sp_0}$ is considered to be strongly negative.
The inverse of the map in Theorem \ref{thm8} is denoted
by $\sigma \mapsto \rm{Jord}(\sigma)$.

Irreducible negative unramified representations can be constructed from
irreducible
strongly negative unramified representations of smaller rank groups as follows.

\begin{thm}[Thereom 5-10, \cite{MT11}]\label{thm9}
For any sequence of pairs $(\chi_1, n_1), \ldots, (\chi_t, n_t)$ with
$\chi_i$ being unramified unitary characters of $F_v^*$ and
$n_i \in \BZ_{\geq 1}$, for $1 \leq i \leq t$,
and for a strongly negative representation $\sigma_{sn}$ of $\Sp_{2n'}(F_v)$
with $\sum_{i=1}^t n_i + n' = n$,
the unique irreducible unramified subquotient
of the following induced representation
\begin{equation}\label{thm9equ1}
\chi_1 ({\det}_{{n_1}}) \times \cdots \times \chi_t ({\det}_{{n_t}}) \rtimes \sigma_{sn}
\end{equation}
is negative and it is a subrepresentation.

Conversely, any irreducible negative
unramified representation
$\sigma_{neg}$ of $\Sp_{2n}(F_v)$ can be obtained from the above construction.
The data $(\chi_1, n_1), \ldots, (\chi_t, n_t)$ and $\sigma_{sn}$
are unique, up to permutations and taking inverses of $\chi_i$'s.
\end{thm}

For any irreducible negative unramified representation $\sigma_{neg}$
with data in Theorem \ref{thm9}, we define
\begin{equation*}
\rm{Jord}(\sigma_{neg}) = \rm{Jord}(\sigma_{sn}) \cup \{(\chi_i,n_i),
(\chi_i^{-1},n_i) | 1 \leq i \leq t\}.
\end{equation*}
By Corollary 3.8 of \cite{M07},
any irreducible negative representation is unitary. In
particular, we have the following
\begin{cor}
Any irreducible negative unramified representation of $\Sp_{2n}(F_v)$ is
unitary.
\end{cor}

To describe the general unramified unitary dual, we need to recall the following definition.

\begin{defn}[Definition 5-13, \cite{MT11}] \label{def1}
Let $\CM^{unr}(n)$ be the set of pairs
$(\textbf{e}, \sigma_{neg})$, where $\textbf{e}$ is a
multiset of triples $(\chi, m, \alpha)$ with $\chi$ being an unramified
unitary character of $F_v^*$, $m \in \BZ_{>0}$ and $\alpha \in \BR_{>0}$, and
$\sigma_{neg}$ is an irreducible negative unramified
representation of $\Sp_{2n''}(F_v)$, having the property that $\sum_{(\chi, m)} m \cdot \#{\textbf{e}(\chi,m)} + n'' = n$
with $\textbf{e}(\chi,m) = \{\alpha | (\chi, m, \alpha) \in \textbf{e}\}$.
Note that $\alpha \in \textbf{e}(\chi,m)$ is counted with multiplicity.

Let $\CM^{u,unr}(n)$ be the subset of $\CM^{unr}(n)$ consisting of pairs
$(\textbf{e}, \sigma_{neg})$, which satisfy the following conditions:
\begin{enumerate}
\item If $\chi^2 \neq 1_{\GL_1}$, then $\textbf{e}(\chi,m) = \textbf{e}(\chi^{-1},m)$,
and $0 < \alpha < \frac{1}{2}$, for all $\alpha \in \textbf{e}(\chi,m)$.
\item If $\chi^2 = 1_{\GL_1}$, and $m$ is even, then
$0 < \alpha < \frac{1}{2}$, for all $\alpha \in \textbf{e}(\chi,m)$.
\item If $\chi^2 = 1_{\GL_1}$, and $m$ is odd, then
$0 < \alpha < 1$, for all $\alpha \in \textbf{e}(\chi,m)$.
\end{enumerate}
Write elements in $\textbf{e}(\chi,m)$ as follows:
\begin{equation*}
0 < \alpha_1 \leq \cdots \leq \alpha_k \leq \frac{1}{2} < \beta_1 \leq
\cdots \leq \beta_l < 1,
\end{equation*}
with $k,l\in \BZ_{\geq 0}$. They satisfy the following conditions:

(a) If $(\chi, m) \notin \rm{Jord}(\sigma_{neg})$,
then $k+l$ is even.

(b) If $k \geq 2$, $\alpha_{k-1} \neq \frac{1}{2}$.

(c) If $l \geq 2$, then $\beta_1 < \beta_2 < \cdots < \beta_l$.

(d) $\alpha_i + \beta_j \neq 1$, for any $1 \leq i \leq k$,
$1 \leq j \leq l$.

(e) If $l \geq 1$, then $\#\{i | 1-\beta_1 < \alpha_i  \leq \frac{1}{2}\}$
is even.

(f) If $l \geq 2$, then $\#\{i | 1-\beta_{j+1} < \alpha_i < \beta_j\}$
is odd, for any $1 \leq j \leq l-1$.
\end{defn}

\begin{thm}[Theorem 5-14, \cite{MT11}]\label{thm10}
The map
$$(\textbf{e}, \sigma_{neg}) \mapsto \times_{(\chi, m, \alpha) \in \textbf{e}}
v^{\alpha} \chi({\det}_m) \rtimes \sigma_{neg}$$
defines a one-to-one correspondence between
the set $\CM^{u,unr}(n)$ and the set of equivalence classes of all
irreducible unramified unitary representations of $\Sp_{2n}(F_v)$.
\end{thm}

In Section 4, we will mainly consider the following two types of strongly negative unramified unitary representations:

\textbf{Type I}: An irreducible strongly negative unramified unitary representations of $\Sp_{2n}(F_v)$ is called of Type I if it is of the following form:

\begin{align}\label{typeI}
\begin{split}
& \nu^{\frac{m_{l-1}-m_l}{2}} 1_{{\det}_{m_{l-1}+m_l+1}}
\times
\nu^{\frac{m_{l-3}-m_{l-2}}{2}} 1_{{\det}_{m_{l-3}+m_{l-2}+1}}\\
& \times \cdots \times
\nu^{\frac{m_{2}-m_3}{2}} 1_{{\det}_{m_{2}+m_3+1}}
\rtimes 1_{\Sp_{2m_1}}.
\end{split}
\end{align}

\textbf{Type II}: An irreducible strongly negative unramified unitary representations of $\Sp_{2n}(F_v)$ is called of Type II if it is of the following form:

\begin{align}\label{typeII}
\begin{split}
& \nu^{\frac{n_{k-1}-n_k}{2}} \lambda_0({\det}_{n_{k-1}+n_k+1})
\times
\nu^{\frac{n_{k-3}-n_{k-2}}{2}} \lambda_0({\det}_{n_{k-3}+n_{k-2}+1})\\
& \times \cdots \times
\nu^{\frac{n_{1}-n_2}{2}} \lambda_0({\det}_{n_{1}+n_2+1})
\rtimes 1_{\Sp_{0}}.
\end{split}
\end{align}

In Section 5, we will mainly consider the following two types of unramified unitary representations:

\textbf{Type III}: An irreducible unramified unitary representations of $\Sp_{2n}(F_v)$ is called of Type III if it is of the following form:

\begin{equation}\label{typeIII}
\sigma = \times_{(\chi, m, \alpha) \in \textbf{e}}
v^{\alpha} \chi({\det}_m) \rtimes \sigma_{neg} \leftrightarrow (\textbf{e}, \sigma_{neg}),
\end{equation}
where $\sigma_{neg}$ is the unique irreducible negative unramified  subrepresentation
of the following induced representation
\begin{equation*}
\chi_1 ({\det}_{{n_1}}) \times \cdots \times \chi_t ({\det}_{{n_t}}) \rtimes \sigma_{sn},
\end{equation*}
with $\sigma_{sn}$ being the unique strongly negative unramified constituent of the following induced representation corresponding to $\rm{Jord}(\sigma_{sn})$ of the form \eqref{sec7equ1}:
\begin{align*}
\begin{split}
& \nu^{\frac{m_{l-1}-m_l}{2}} 1_{{\det}_{m_{l-1}+m_l+1}}
\times
\nu^{\frac{m_{l-3}-m_{l-2}}{2}} 1_{{\det}_{m_{l-3}+m_{l-2}+1}}\\
& \times \cdots \times
\nu^{\frac{m_{2}-m_3}{2}} 1_{{\det}_{m_{2}+m_3+1}}
\rtimes 1_{\Sp_{2m_1}}.
\end{split}
\end{align*}

\textbf{Type IV}: An irreducible unramified unitary representations of $\Sp_{2n}(F_v)$ is called of Type IV if it is of the following form:

\begin{equation}\label{typeIV}
\sigma = \times_{(\chi, m, \alpha) \in \textbf{e}}
v^{\alpha} \chi({\det}_m) \rtimes \sigma_{neg} \leftrightarrow (\textbf{e}, \sigma_{neg}),
\end{equation}
where $\sigma_{neg}$ is the unique irreducible negative unramified  subrepresentation
of the following induced representation
\begin{equation*}
\chi_1 ({\det}_{{n_1}}) \times \cdots \times \chi_t ({\det}_{{n_t}}) \rtimes \sigma_{sn},
\end{equation*}
with $\sigma_{sn}$ being the unique strongly negative unramified constituent of the following induced representation corresponding to $\rm{Jord}(\sigma_{sn})$ of the form \eqref{sec7equ1}:
\begin{align*}
\begin{split}
& \nu^{\frac{n_{k-1}-n_k}{2}} \lambda_0({\det}_{n_{k-1}+n_k+1})
\times
\nu^{\frac{n_{k-3}-n_{k-2}}{2}} \lambda_0({\det}_{n_{k-3}+n_{k-2}+1})\\
& \times \cdots \times
\nu^{\frac{n_{1}-n_2}{2}} \lambda_0({\det}_{n_{1}+n_2+1})
\rtimes 1_{\Sp_{0}}.
\end{split}
\end{align*}


\subsection{Arthur Parameters and Unramified Local Components}

For a given global Arthur parameter $\psi \in \widetilde{\Psi}_2(\Sp_{2n})$, $\wt{\Pi}_{\psi}(\varepsilon_{\psi})$
is the corresponding automorphic $L^2$-packet.
It is clear that the irreducible unramified representations determined by
the
local Arthur parameter $\psi_v$ at almost all unramified local places $v$ of
$F$ determine the unramified local components of $\pi$ for all members
$\pi \in \wt{\Pi}_{\psi}(\varepsilon_{\psi})$.
We fix one of the members, $\pi \in \wt{\Pi}_{\psi}(\varepsilon_{\psi})$,
and are going to describe the unramified local component $\pi_v$, where $v$ is a finite place of $F$ such that the local Arthur parameter
$$
\psi_v=\psi_{1,v}\boxplus\psi_{2,v}\boxplus\cdots\boxplus\psi_{r,v}
$$
is unramified, i.e. $\tau_{i,v}$ for $i=1,2,\cdots,r$ are all unramified.

Rewrite the global Arthur parameter $\psi$ as follows:
\begin{equation}
\psi = [\boxplus_{i=1}^k (\tau_i, 2b_i)] \boxplus [\boxplus_{j=k+1}^{k+l}
(\tau_j, 2b_j+1)] \boxplus [\boxplus_{s=k+l+1}^{k+l+2t+1} (\tau_s, 2b_s+1)],
\end{equation}
where $\tau_i \in \CA_{\cusp}(\GL_{2a_i})$ is of symplectic type for $1 \leq i \leq k$, $\tau_j \in \CA_{\cusp}(\GL_{2a_j})$ and $\tau_s \in \CA_{\cusp}(\GL_{2a_s+1})$ are of orthogonal type for $k+1 \leq j \leq k+l$ and $k+l+1 \leq s \leq k+l+2t+1$. Let $I=\{1,2,\ldots,k\}$, $J=\{k+1, k+2, \ldots, k+l\}$, and $S=\{k+l+1, k+l+2, \ldots, k+l+2t+1\}$. Let $J_1$ be the subset of $J$ such that $\omega_{\tau_{j,v}}=1$, and $J_2=J \bs J_1$, that is, for $j \in J_2$, $\omega_{\tau_{j,v}}=\lambda_0$.
Let $S_1$ be the subset of $S$ such that $\omega_{\tau_{s,v}}=1$, and $S_2=S \bs S_1$, that is, for $s \in S_2$, $\omega_{\tau_{s,v}}=\lambda_0$.
From the definition of Arthur parameters, we can easily see that $\#\{J_2\} \cup \#\{S_2\}$ is even, which implies that $\#\{J_2\} \cup \#\{S_1\}$ is odd. The local unramified Arthur parameter $\psi_v$ has the following structures:
\begin{itemize}
\item For $i \in I$,
$$\tau_{i,v} = \times_{q=1}^{a_i} \nu^{\beta^i_q} \chi^i_q
\times_{q=1}^{a_i} \nu^{-\beta^i_q} \chi_q^{i,-1},$$
where $0 \leq \beta^i_q < \frac{1}{2}$, for $1 \leq q \leq a_i$, and $\chi^i_q$'s are unramified unitary
characters of $F_v^*$.

\item For $j \in J_1$,
$$\tau_{j,v} = \times_{q=1}^{a_j} \nu^{\beta^j_q} \chi^j_q
\times_{q=1}^{a_j} \nu^{-\beta^j_q} \chi_q^{j,-1},$$
where $0 \leq \beta^j_q < \frac{1}{2}$, for $1 \leq q \leq a_j$, and $\chi^j_q$'s are unramified unitary
characters of $F_v^*$.

\item For $j \in J_2$,
$$\tau_{j,v} = \times_{q=1}^{a_j-1} \nu^{\beta^j_q} \chi^j_q \times
\lambda_0 \times 1_{GL_1}
\times_{q=1}^{a_j-1} \nu^{-\beta^j_q} \chi_q^{j,-1},$$
where $0 \leq \beta^j_q < \frac{1}{2}$, for $1 \leq q \leq a_j$, and $\chi^j_q$'s are unramified unitary
characters of $F_v^*$.

\item For $s \in S_1$,
$$\tau_{s,v} = \times_{q=1}^{a_s} \nu^{\beta^s_q} \chi^s_q \times
1_{GL_1}
\times_{q=1}^{a_s} \nu^{-\beta^s_q} \chi_q^{s,-1},$$
where $0 \leq \beta^s_q < \frac{1}{2}$, for $1 \leq q \leq a_s$, and $\chi^s_q$'s are unramified unitary
characters of $F_v^*$.

\item For $s \in S_2$,
$$\tau_{s,v} = \times_{q=1}^{a_s} \nu^{\beta^s_q} \chi^s_q \times
\lambda_0
\times_{q=1}^{a_s} \nu^{-\beta^s_q} \chi_q^{s,-1},$$
where $0 \leq \beta^s_q < \frac{1}{2}$, for $1 \leq q \leq a_s$, and $\chi^s_q$'s are unramified unitary
characters of $F_v^*$.
\end{itemize}

We define
\begin{align*}
\rm{Jord}_1
= \ &\{(\lambda_0, 2b_j+1), j \in J_2; (\lambda_0, 2b_s+1), s \in S_2;\\
& (1_{GL_1}, 2b_j+1), j \in J_2; (1_{GL_1}, 2b_s+1), s \in S_1\}.
\end{align*}
Note that $\rm{Jord}_1$ is a multi-set.
Let $\rm{Jord}_2$ be set consists of different Jordan blocks with odd multiplicities in
$\rm{Jord}_1$. Then $\rm{Jord}_2$ has the form of \eqref{sec7equ1}, and by Theorem \ref{thm8}, there is a corresponding irreducible strongly negative unramified representation $\sigma_{sn}$.

Then we define the following Jordan blocks:
\begin{align*}
\rm{Jord}_I & = \{(\chi_q^i, 2b_i), (\chi_q^{i,-1}, 2b_i), i \in I, 1 \leq q \leq a_i, \beta_q^i=0\},\\
\rm{Jord}_{J_1} & = \{(\chi_q^j, 2b_j+1), (\chi_q^{j,-1}, 2b_j+1), j \in J_1, 1 \leq q \leq a_j, \beta_q^j=0\},\\
\rm{Jord}_{J_2} & = \{(\chi_q^j, 2b_j+1), (\chi_q^{j,-1}, 2b_j+1), j \in J_2, 1 \leq q \leq a_j-1, \beta_q^j=0\},\\
\rm{Jord}_{S_1} & = \{(\chi_q^s, 2b_s+1), (\chi_q^{s,-1}, 2b_s+1), s \in S_1, 1 \leq q \leq a_s, \beta_q^s=0\},\\
\rm{Jord}_{S_2} & = \{(\chi_q^s, 2b_s+1), (\chi_q^{s,-1}, 2b_s+1), s \in S_2, 1 \leq q \leq a_s, \beta_q^s=0\}.
\end{align*}
Finally, we define
$$\rm{Jord}_3 = (\rm{Jord}_1 \bs \rm{Jord}_2) \cup \rm{Jord}_I \cup \rm{Jord}_{J_1} \cup \rm{Jord}_{J_2} \cup \rm{Jord}_{S_1} \cup \rm{Jord}_{S_2}.$$ By Theorem \ref{thm9}, corresponding to data $\rm{Jord}_3$ and $\sigma_{sn}$, there is an irreducible negative unramified presentation $\sigma_{neg}$.

Let
\begin{align*}
\textbf{e}_I & = \{(\chi_q^i, 2b_i, \beta_q^i), i \in I, 1 \leq q \leq a_i, \beta_q^i>0\},\\
\textbf{e}_{J_1} & = \{(\chi_q^j, 2b_j+1, \beta_q^j), j \in J_1, 1 \leq q \leq a_j, \beta_q^j>0\},\\
\textbf{e}_{J_2} & = \{(\chi_q^j, 2b_j+1, \beta_q^j), j \in J_2, 1 \leq q \leq a_j-1, \beta_q^j>0\},\\
\textbf{e}_{S_1} & = \{(\chi_q^s, 2b_s+1, \beta_q^s), s \in S_1, 1 \leq q \leq a_s, \beta_q^s>0\},\\
\textbf{e}_{S_2} & = \{(\chi_q^s, 2b_s+1, \beta_q^s), s \in S_2, 1 \leq q \leq a_s, \beta_q^s>0\}.
\end{align*}
Then we define
\begin{equation}\label{sete}
\textbf{e}=\textbf{e}_I \cup \textbf{e}_{J_1}\cup \textbf{e}_{J_2}
\cup \textbf{e}_{S_1} \cup \textbf{e}_{S_2}.
\end{equation}

Since the unramified component $\pi_v$ is unitary, we must have that
$(\textbf{e}, \sigma_{neg}) \in \CM^{u,unr}(n)$, and $\pi_v$ is exactly the irreducible unramified unitary representation $\sigma$ of $\Sp_{2n}(F_v)$ which corresponds to $(\textbf{e}, \sigma_{neg})$ as in Theorem \ref{thm10}.

\begin{rmk}
(1) If $\sigma$ is an irreducible unramified unitary representation of $\Sp_{2n}(F_v)$ corresponding to the pair $(\textbf{e}, \sigma_{neg}) \in \CM^{u,unr}(n)$, then the orbit $\check{\CO}$ corresponding to $\sigma$ in \cite{Bar10} is given by the following partition:
$$[(\prod_{j=1}^t n_j^2) (\prod_{(\chi,m,\alpha) \in \textbf{e}} m^2) (\prod_{i=1}^k (2n_i+1))(\prod_{i=1}^l (2m_i+1))].$$

(2) In Section 6.2, we will show that given an Arthur parameter $\psi$, there are infinitely many finite local
places $v$ such that $\psi_v$ are unramified and the central characters of $\tau_{i,v}$ are trivial. It follows that for any $\pi \in \wt{\Pi}_{\psi}(\epsilon_{\psi})$, there is such a finite local place $v$, such that $\pi_v$ is an irreducible unramified unitary representation of \textbf{Type III} as in \ref{typeIII}.
This is a key step in the proof of Theorem \ref{main}.

For such $\pi_v$ as in \ref{typeIII}, the orbit $\check{\CO}$ corresponding to $\sigma$ in \cite{Bar10} is given by the following partition:
$$[(\prod_{j=1}^t n_j^2) (\prod_{(\chi,m,\alpha) \in \textbf{e}} m^2) (\prod_{i=1}^l (2m_i+1))],$$
which actually turns out to be $\ul{p}(\psi)$.
Then, we will show that $\pi$ has no non-zero Fourier coefficients attached to any symplectic partition $\ul{p}$ which is bigger than
the Barbasch-Vogan duality partition $\eta_{{\frak{g}^\vee,\frak{g}}}(\ul{p}(\psi))$. This proves Theorem \ref{main}.
\end{rmk}


\section{Vanishing of Certain Twisted Jacquet Modules}

For an irreducible automorphic representation $\pi$ of $\Sp_{2n}(\BA)$, we write $\pi=\otimes_v\pi_v$, the restricted tensor product
decomposition. The (global) Fourier coefficients on $\pi$ induce the corresponding (local) twisted Jacquet modules of $\pi_v$ for
each local place $v$ of $F$. It is clear that if $\pi_v$ has no nonzero such twisted Jacquet modules at one local place $v$,
then $\pi$ has no nonzero corresponding (global) Fourier coefficients. We consider the local twisted Jacquet modules at
an unramified local place of $\pi$, the structure of which implies the vanishing property of the corresponding Fourier coefficients.

To simplify notation, we denote, in this section,
by $\pi$ for an irreducible admissible representation of $G_n(F_v)$, where $v$ is a finite place of $F$.

Recall from \cite{JL15a} that
given any symplectic partition $\ul{p}$ of $2n$ and a datum $\ul{a}$, there is a unipotent subgroup $V_{\ul{p},2}$ and a character $\psi_{\ul{p},\ul{a}}$.
Given any irreducible admisslbe representation $\pi$ of $G_n(F_v)$, let
$J_{V_{\ul{p},2}, \psi_{\ul{p},\ul{a}}}(\pi)$ be the twisted Jacquet module of $\pi$
with respect to the unipotent subgroup $V_{\ul{p},2}$ and the character $\psi_{\ul{p},\ul{a}}$.

In principle, we are mainly interested in irreducible unramified unitary representations which are described in Section 2.1.
However, in this section, we consider
the following more general induced
representations of $G_n(F_v)$:
\begin{equation}\label{sec4equ1}
\pi = \Ind_{P_{m_1, \ldots, m_k}(F_v)}^{G_n(F_v)} \mu_1({\det}_{m_1}) \otimes \cdots \otimes
\mu_k({\det}_{m_k}) \otimes 1_{G_{m_0}},
\end{equation}
where $m_0 = n-\sum_{i=1}^k m_i \geq 0$,
$P_{m_1, \ldots, m_k}$ is a standard parabolic subgroup of $G_n$ with
Levi subgroup isomorphic to $\GL_{m_1} \times \cdots \GL_{m_k} \times G_{m_0}$
and $\mu_i$'s are quasi-characters of $F_v^*$.

We prove the following vanishing properties of certain
twisted Jacquet modules of $\pi$.

\begin{lem}\label{lemv1}
For $\pi$ as in \eqref{sec4equ1}, the following statements hold.
\begin{enumerate}
\item[(1)] $J_{{\psi_{r-1}^{\alpha}}}(\pi):=J_{V_{[(2r)1^{2n-2r}],2}, \psi_{[(2r)1^{2n-2r}], \alpha}}(\pi) \equiv 0$,
for any square class $\alpha \in F_v^*/(F_v^*)^2$ and any $r \geq k+1$.
\item[(2)] $J_{\psi_{(2r+1)^2}}(\pi):=J_{V_{[(2r+1)^2 1^{2n-4r-2}],2}, \psi_{[(2r+1)^2 1^{2n-4r-2}]}}(\pi) \equiv 0$,
\begin{itemize}
\item for any $r \geq k$ if $m_0 =0$, or, if $m_0 > 0$ and $m_i = 1$ for some $1 \leq i \leq k$, assuming that $2(2k+1) \leq 2n$;
\item for any $r \geq k+1$ if $m_0 >0$ and $m_i > 1$ for any $1 \leq i \leq k$, assuming that $2(2k+3) \leq 2n$.
\end{itemize}
\end{enumerate}
\end{lem}

\begin{proof}
The idea of the proof of Part (1) is similar to that of Key Lemma 3.3 of \cite{GRS05}.

By the adjoint relation between parabolic induction and the twisted Jacquet module, we consider
$P_{m_1, \ldots, m_k} \bs G_n / V_{[(2r)1^{2n-2r}],2}$, the double coset
decomposition of $G_n$. Using the generalized Bruhat decomposition, the representatives of these
double cosets can be chosen to be
elements of the following form:
$
\gamma = \omega u_{\omega},
$
with $\omega \in W(P_{m_1, \ldots, m_k}) \bs W(G_n)$, where
$W(G_n)$ is the Weyl group of $G_n$ and
$W(P_{m_1, \ldots, m_k})$ is the Weyl group of $P_{m_1, \ldots, m_k}$, and with
$u_{\omega}$ being in the standard maximal unipotent subgroup of $\Sp_{2n-2r+2}$,
which is embedded in $G_n$ as $I_{r-1} \times \Sp_{2n-2r+2}$ in the Levi subgroup $\GL_{r-1}\times\Sp_{2n-2r+2}$.
We identify $u_{\omega}$ with its embedding image.

We show that there is no admissible double coset,
i.e., for any representative $\gamma = \omega u_{\omega}$,
there exists $v \in V_{[(2r)1^{2n-2r}],2}$, such that
$\gamma v \gamma^{-1} \in P_{m_1, \ldots, m_k}$, but
$\psi_{[(2r)1^{2n-2r}], \alpha}(v) \neq 1$.

Let $\alpha_i = e_i - e_{i+1}$, for $i=1,2,\ldots,r-1$,
and $\alpha_{r} = e_{r}+e_{r}$ be some positive roots.
By definition, $\psi_{[(2r)1^{2n-2r}], \alpha}$
is non-trivial on the corresponding one-dimensional root subgroup $X_{\alpha_i}$
for $i=1,2,\ldots, r$, but is trivial on the root subgroup
corresponding to any other positive root.
Hence it is enough to show that for any representative $\gamma$,
there is at least one $1 \leq i \leq r$,
such that $\gamma X_{\alpha_i}(x) \gamma^{-1} \in P_{m_1, \ldots, m_k}$.

Note that
$\omega u_{\omega} X_{\alpha_i}(x) (\omega u_{\omega})^{-1} =
\omega (u_{\omega} X_{\alpha_i}(x) u_{\omega}^{-1}) \omega^{-1} = \omega X_{\alpha_i}(x) \omega^{-1}$
for any $i \neq r-1$. For $i = r-1$, $u = u_{\omega}^{-1} X_{\alpha_{r}}(x) u_{\omega} \in V_{[(2r)1^{2n-2r}],2}$,
and $\omega u_{\omega} u (\omega u_{\omega})^{-1} =
\omega (u_{\omega} u u_{\omega}^{-1}) \omega^{-1} = \omega X_{\alpha_{r}}(x) \omega^{-1}$.
Therefore, it remains to show that for any Weyl element $\omega \in W(P_{m_1, \ldots, m_k}) \bs W(G_n)$,
there is at least one $1 \leq i \leq r$,
such that $\omega X_{\alpha_i}(x) \omega^{-1} \in P_{m_1, \ldots, m_k}$.

Let $N_{m_1, \ldots, m_k}$ be the unipotent radical of $P_{m_1, \ldots, m_k}$, and
$\ol{N}_{m_1, \ldots, m_k}$ be its opposite.
Assume that there is an $\omega \in W(P_{m_1, \ldots, m_k}) \bs W(G_n)$, such that
for any $1 \leq i \leq r$,
$\omega X_{\alpha_i}(x) \omega^{-1} \in \ol{N}_{m_1, \ldots, m_k}$. This will lead us to a contradiction.

We separate the numbers $\{1, \ldots, \sum_{i=1}^{k}m_i\}$ into the following chunks
of indices:
$I_j = \{\sum_{i=1}^{j-1}m_i+1, \sum_{i=1}^{j-1}m_i+2, \ldots, \sum_{i=1}^{j}m_i\}$, for $1 \leq j \leq k$.
By assumption, $\omega X_{\alpha_i}(x) \omega^{-1} \in \ol{N}_{m_1, \ldots, m_k}$ for any
$1 \leq i \leq r$, where $\alpha_i = e_i - e_{i+1}$ if $i=1,2,\ldots,r-1$, and
$\alpha_{r} = e_{r}+e_{r}$.
There must exist a sequence of numbers
$1 \leq i_1 < i_2 < \cdots < i_{r-1} <i_{r} \leq n$, such that
$\omega(e_{s}) = -e_{i_s}$, for $s = 1, 2, \ldots, r$.

We assume that $i_s \in I_{j_s}$ for $s = 1, 2, \ldots, r$.
We claim that $j_1 < j_2 < \cdots < j_{r}$. Indeed,
we have $j_1 \leq j_2 \leq \cdots \leq j_{r}$.
If $j_s = j_{s+1}$ for some $s\in\{1, 2, \ldots, r-1\}$,
then
$$
\omega X_{\alpha_s} \omega^{-1} = \omega X_{e_s-e_{s+1}} \omega^{-1} = X_{e_{i_{s+1}}-e_{i_s}} \subset P_{m_1, \ldots, m_k},
$$
which is a contradiction. This justifies the claim. On the other hand,
the condition that $j_1 < j_2 < \cdots < j_{r}$ will lead to
a contradiction, since we just have $k$ different chunks of indices,
and $r \geq k+1$.

This completes the proof of Part (1).

Next, we prove Part (2). For the partition $[(2r+1)^2 1^{2n-4r-2}]$, the corresponding
one-dimensional toric subgroup $\CH_{[(2r+1)^2 1^{2n-4r-2}]}$ consists of
elements as follows
\begin{equation}\label{sec4equ4}
\diag(t_1^{2r}, t_1^{2r-2}, \ldots, t_1^{-2r}; I_{2n-4r-2}, t_2^{2r}, t_2^{2r-2}, \ldots, t_2^{-2r}).
\end{equation}
Note that here actually $t_1 = t_2 = t$, we just
label them to distinguish their positions.

Let $\omega_1$ be a Weyl element sending the one-dimensional toric subgroup $\CH_{[(2r+1)^2 1^{2n-4r-2}]}$ to the following one-dimensional toric subgroup
\begin{equation}\label{sec4equ5}
\{\diag(T; I_{n-2r-1}; t_1^0, t_2^0; I_{n-2r-1}, T^*)\},
\end{equation}
where $T=\diag(t_1^{2r}, t_2^{2r}; t_1^{2r-2}, t_2^{2r-2}; \ldots, t_1^{2}, t_2^{2})$.
Then it is easy to see that $J_{\psi_{(2r+1)^2}}(\pi) \neq 0$ if and only if $\pi$ has
a non-zero twisted Jacquet module with respect to
$U := \omega_1 V_{[(2r+1)^2 1^{2n-4r-2}],2} \omega_1$
and $\psi_{U}$, which is defined by
$$\psi_U(u):=\psi_{[(2r+1)^2 1^{2n-4r-2}]}(\omega_1^{-1} u \omega_1).$$
Hence we have to show that
$J_{U, \psi_U}(\pi)=0$.
Note that $U$ is actually the unipotent radical of the parabolic
subgroup with Levi $M$ isomorphic to
$\GL_2 \times \cdots \times \GL_2 \times G_{n-2r}$ (with $r$-copies of $\GL_2$).

As in the proof of Part (1), we consider the double coset
decomposition
$P_{m_1, \ldots, m_k} \bs G_n / U$.
By Bruhat decomposition, the representatives of these
double cosets can be chosen to be
elements of the following form: $\gamma = \omega u_{\omega}$,
with $\omega \in W(P_{m_1, \ldots, m_k}) \bs W(G_n)$ and
$u_{\omega}$ in the standard maximal unipotent subgroup of $M$.
We will show that there is no admissible double coset,
i.e., for any representative $\gamma = \omega u_{\omega}$,
there exists $u \in U$, such that
$\gamma u \gamma^{-1} \in P_{m_1, \ldots, m_k}$, but
$\psi_{U}(u) \neq 1$.

Define for now that $\alpha_i = \omega_1(e_i-e_{i+1})$ for $1 \leq i \leq 2r$.
We show that for any representative $\gamma$,
there is at least one $1 \leq i \leq 2r$
such that $\gamma X_{\alpha_i}(x) \gamma^{-1} \in P_{m_1, \ldots, m_k}$.

Note that for any $u_{\omega} \in M$ and any $1 \leq i \leq 2r$,
$u = u_{\omega}^{-1} X_{\alpha_i}(x) u_{\omega} \in U$ since
$M$ normalizes $U$. Then
$\omega u_{\omega} u (\omega u_{\omega})^{-1} = \omega X_{\alpha_i}(x) \omega^{-1}$.
Hence we just have to show that for an $\omega \in W(P_{m_1, \ldots, m_k}) \bs W(G_n)$,
there is at least one $1 \leq i \leq 2r$,
such that $\omega X_{\alpha_i}(x) \omega^{-1} \in P_{m_1, \ldots, m_k}$.
As in part (1), we prove this by contradiction, assuming that there is an
$\omega \in W(P_{m_1, \ldots, m_k}) \bs W(G_n)$, such that
for any $1 \leq i \leq 2r$,
$\omega X_{\alpha_i}(x) \omega^{-1} \in \ol{N}_{m_1, \ldots, m_k}$.

If $m_0=0$, we separate the numbers $\{1, \ldots, n, -n, -n+1, \ldots, -1\}$ into the following $2k$ chunks
of indices:
$I_j = \{\sum_{i=1}^{j-1}m_i+1, \sum_{i=1}^{j-1}m_i+2, \ldots, \sum_{i=1}^{j}m_i\}$, if $1 \leq j \leq k$;
and $I_j = \{-\sum_{i=1}^{2k-j+1} m_i, -\sum_{i=1}^{2k-j+1} m_i+1, \ldots,
-\sum_{i=1}^{2k-j} m_i-1\}$,
if $k+1 \leq j \leq 2k$.

If $m_0>0$, we separate the numbers $\{1, \ldots, n, -n, -n+1, \ldots, -1\}$ into the following $2k+1$ chunks
of indices:
$I_j = \{\sum_{i=1}^{j-1}m_i+1, \sum_{i=1}^{j-1}m_i+2, \ldots, \sum_{i=1}^{j}m_i\}$, if $1 \leq j \leq k$;
$$I_{k+1} = \{\sum_{i=1}^{k}m_i+1, \sum_{i=1}^{k}m_i+2, \ldots, n, -n, -n+1, \ldots, -\sum_{i=1}^{k}m_i-1\};$$
and $I_{j+1} = \{-\sum_{i=1}^{2k-j+1} m_i, -\sum_{i=1}^{2k-j+1} m_i+1, \ldots,
-\sum_{i=1}^{2k-j} m_i-1\}$,
if $k+1 \leq j \leq 2k$.

By assumption, $\omega X_{\alpha_i}(x) \omega^{-1} \in \ol{N}_{m_1, \ldots, m_k}$ for any
$1 \leq i \leq 2r$, where $\alpha_i = \omega_1(e_i - e_{i+1})$ with $i=1,2,\ldots,2r$.
There must exist a sequence of numbers
$\{i_1, i_2, \ldots, i_{2r+1}\}$ with $i_s \in I_{j_s}$
and $j_1 \leq j_2 \leq \cdots \leq j_{2r+1}$,
such that
$\omega(\omega_1(e_{s})) = f_{i_{2r+2-s}}$ for $s = 1, 2, \ldots, 2r+1$,
where  $f_t:=e_t$, if $t>0$,
and $f_t := -e_{-t}$, if $t<0$.
Using similar augments as in the proof of Part (1), we have that
$j_1 < j_2 < \cdots < j_{2r+1}$.


Assuming that $2(2k+1)\leq 2n$.
If $m_0=0$ and $r \geq k$, then
the condition that $j_1 < j_2 < \cdots < j_{2r+1}$ will lead to
a contradiction, since we just have $2k$ different chunks of indices.
If $m_0 >0$ and $m_i=1$ for some $1 \leq i \leq k$, and $r \geq k+1$, then the condition that $j_1 < j_2 < \cdots < j_{2r+1}$ will also lead to
a contradiction, since we just have $2k+1$ different chunks of indices.
If $m_0 >0$ and $m_i=1$ for some $1 \leq i \leq k$, and $r=k$, then $j_s=s$, that is $i_s \in I_s$, $1 \leq s \leq 2k+1$. This easily implies that $\#(I_s) \geq 2$ for any $1 \leq s \leq 2k+1$, that is, $m_s \geq 2$ for any $1 \leq s \leq 2k+1$. This is a contradiction since $m_i=1$ for some $1 \leq i \leq k$.

Assuming that $2(2k+3)\leq 2n$.
If $m_0>0$ and $m_i >1$ for any $1 \leq i \leq k$, and $r \geq k+1$, then
the condition that $j_1 < j_2 < \cdots < j_{2r+1}$ will also lead to
a contradiction, since  we just have $2k+1$ different chunks of indices.

This completes the proof of part (2) and hence the proof the lemma.
\end{proof}

Lemma \ref{lemv1} is also true
for the double cover of $\Sp_{2n}(F_v)$, with exactly the same proof. We state the result as follows with proof omitted.

Let
\begin{equation}\label{sec4equ2}
\wt{\pi} = \Ind_{\wt{P}_{m_1, \ldots, m_k}(F_v)}^{\wt{\Sp}_{2n}(F_v)} \mu_{\psi} \mu_1({\det}_{m_1}) \otimes \cdots \otimes
\mu_k({\det}_{m_k}) \otimes 1_{\wt{\Sp}_{2m_0}(F_v)},
\end{equation}
where $m_0 = n-\sum_{i=1}^k m_i \geq 0$,
$\wt{P}_{m_1, \ldots, m_k}(F_v)$ is the pre-image of the parabolic $P_{m_1, \ldots, m_k}(F_v)$
in $\wt{\Sp}_{2n}(F_v)$,
$\mu_i$'s are quasi-characters of $F_v^*$, and $\mu_{\psi}$ is defined as in (6.1) of \cite{GRS11}.

\begin{lem}\label{lemv2}
For $\wt{\pi}$ as in \eqref{sec4equ2}, the followings hold.
\begin{enumerate}
\item[(1)] $J_{\psi_{r-1}^{\alpha}}(\wt{\pi}):=J_{V_{[(2r)1^{2n-2r}],2}, \psi_{[(2r)1^{2n-2r}], \alpha}}(\wt{\pi}) \equiv 0$,
for any square class $\alpha \in F_v^*/(F_v^*)^2$ and any $r \geq k+1$.
\item[(2)] $J_{\psi_{(2r+1)^2}}(\wt{\pi}):=J_{V_{[(2r+1)^2 1^{2n-4r-2}],2}, \psi_{[(2r+1)^2 1^{2n-4r-2}]}}(\wt{\pi}) \equiv 0$,
\begin{itemize}
\item for any $r \geq k$ if $m_0 =0$, or, if $m_0 > 0$ and $m_i = 1$ for some $1 \leq i \leq k$, assuming that $2(2k+1) \leq 2n$;
\item for any $r \geq k+1$ if $m_0 >0$ and $m_i > 1$ for any $1 \leq i \leq k$, assuming that $2(2k+3) \leq 2n$.
\end{itemize}
\end{enumerate}
\end{lem}

\begin{rmk}
When $\alpha=1$ and $m_0=0$,  Part (1) of Lemmas \ref{lemv1} and \ref{lemv2} have already been
proved in Theorem 6.3 of \cite{GRS11}.
\end{rmk}


\section{Vanishing of Certain Fourier Coefficients: Strongly Negative Case}

In this and next sections, we  characterize the vanishing property of Fourier coefficients
for certain irreducible automorphic representations $\pi$, based on the information
of $\pi_v$, where $v$ is any finite place of $F$ such that $\pi_v$ is unramified.

First, we recall some definitions and results from \cite{CM93}.
A symplectic partition is called {\bf special}
if it has an even number of even parts between any two consecutive
odd ones and an even number of even parts greater than the largest odd part.
By Theorem 2.1 of \cite{GRS03} and Corollary 4.2 of \cite{JL15a},
for an irreducible automorphic representation $\pi$ of $\Sp_{2n}(\BA)$, $\frak{p}^m(\pi)$ consists of special symplectic partitions of $2n$.

Given a partition $\ul{p}$ of $2n$, which is not necessarily symplectic,
the unique largest symplectic
partition which is smaller than $\ul{p}$ is called {\bf the $G$-collapse}
of $\ul{p}$, and is denoted by $\ul{p}_G$ (note that $G=\Sp$).
In general, $\ul{p}_G$ may not be special.
Given a symplectic partition $\ul{p}$ of $2n$,
which is not necessarily special, the smallest
special symplectic partition which is greater than $\ul{p}$
is called {\bf the $G$-expansion} of $\ul{p}$, and is denoted
by $\ul{p}^G$.

Theorem 6.3.8 of \cite{CM93} gives a recipe for passing from
a partition $\ul{p}$ to its $G$-collapse.
Explicitly, given a partition $\ul{p}$ of $2n$, then is
automatically has an even number of
odd parts, but each odd part may not have an even
multiplicity, that is, $\ul{p}$ may not be
symplectic. Assume that its odd parts are
$p_{1} \geq \cdots \geq p_{2r}$,
with multiplicities. Enumerate the indices $i$
with $p_{2i-1} > p_{2i}$ as $i_1 < \cdots < i_t$.
Then, the $G$-collapse of $\ul{p}$ can be obtained by
replacing each pair of parts
$(p_{2i_j-1}, p_{2i_j})$ by $(p_{2i_j-1}-1, p_{2i_j}+1)$,
respectively, for $1 \leq j \leq t$, and leaving
the other parts alone.
For example, for the partition $\ul{p}=[5^34^23^221^3]$,
then its odd parts are $5 \geq 5 \geq 5 >3 \geq 3>1\geq 1 \geq 1$,
and $5 > 3$, $3 > 1$ are two pairs in the series of its odd parts
which are not equal. Then, $\ul{p}_G = [5^2 4^4 2^3 1^2]$,
which is exactly obtained by replacing the pair $(5,3)$ by $(4,4)$,
$(3,1)$ by $(2,2)$, and leaving the other parts alone.

Theorem 6.3.9 of \cite{CM93} gives a recipe for passing from
a symplectic partition $\ul{p}$ to its $G$-expansion.
Explicitly, given a symplectic partition $\ul{p}$ of $2n$,
then by definition, each of its odd parts occurs with even
multiplicity. Assume that $\ul{p}=[p_1 p_2 \cdots p_r]$,
with $p_1 \geq p_2 \geq \cdots \geq p_r >0$. Enumerate the
indices $i$ such that $p_{2i} = p_{2i+1}$ is odd and
$p_{2i-1} \neq p_{2i}$ as $i_1 < \cdots < i_t$. Then
the $G$-expansion of $\ul{p}$ can be obtained by
replacing each pair of parts $(p_{2i_j}, p_{2i_j+1})$
by $(p_{2i_j+1}, p_{2i_j-1})$, respectively,
and leaving the other parts alone.
For example, for the symplectic partition $\ul{p}=[65^243^221^2]$,
which is not special, we have $p_1 \neq p_2 = p_3 =5$,
and $p_7 \neq p_8 = p_9 =1$. Then $\ul{p}^G = [6^2 4^2 3^22^2]$,
which is exactly obtained by replacing the pair $(5, 5)$
by $(6,4)$, $(1,1)$ by $(2,0)$, and leaving the other parts alone.

Following from the definition of Fourier coefficients attached
to composite partitions for the global case in Section 1 of \cite{GRS03},
and also the definition of Fourier-Jacobi module $FJ$ in Section 3.8 of \cite{GRS11},
we can similarly define the Fourier-Jacobi modules
with respect to the composite partitions
like $[(2n_1)1^{2n-2n_1}] \circ [(2n_2)1^{2n-2n_1-2n_2}]$.
Explicitly, given an irreducible admissible representation $\pi$ of
$G_n(F_v)$, we say that $\pi$ has a nonzero Fourier-Jacobi module
with respect to the composite partition
$[(2n_1)1^{2n-2n_1}] \circ [(2n_2)1^{2n-2n_1-2n_2}]$
if the following is nonzero:
first taking the Fourier-Jacobi module $FJ_{\psi_{n_1-1}^{\alpha}} (\pi)$ which is a representation of $\widetilde{G}_{n-n_1}(F_v)$, denoted by $\pi'$,
followed by taking twisted Jacquet module
$J_{\psi_{n_2-1}^{\beta}} (\pi')$,
with $\alpha, \beta \in F_v^* / (F_v^*)^2$.

The following proposition
generalizes Theorem 6.3 of \cite{GRS11}.

\begin{prop}\label{prop1}
The following hold.
\begin{enumerate}
\item Let $\chi_i$, $1 \leq i \leq r$, be characters of $F_v^*$, and $a \in F_v^*$.
Then
 \begin{align}\label{prop1equ1}
\begin{split}
 & FJ_{\psi^a_{k-1}} (\Ind_{P_{m_1, \ldots, m_k}}^{\Sp_{2n}} \nu^{\alpha_1} \chi_1 ({\det}_{{m_1}}) \otimes \cdots \otimes \nu^{\alpha_k} \chi_k({\det}_{{m_k}}))\\
\cong \ & \Ind_{P_{m_1-1, \ldots, m_k-1}}^{\wt{\Sp}_{2n-2k}} \mu_{\psi^{-a}} \nu^{\alpha_1} \chi_1 ({\det}_{{m_1-1}}) \otimes \cdots \otimes \nu^{\alpha_k} \chi_k({\det}_{{m_k-1}}).
\end{split}
\end{align}
\item Let $\chi_i$, $1 \leq i \leq r$, be characters of $F_v^*$, and $a, b \in F_v^*$.
Then
\begin{align}\label{prop1equ2}
\begin{split}
 & FJ_{\psi^b_{k-1}} (\Ind_{P_{m_1, \ldots, m_k}}^{\wt{\Sp}_{2n}} \mu_{\psi^{a}}\nu^{\alpha_1} \chi_1 ({\det}_{{m_1}}) \otimes \cdots \otimes \nu^{\alpha_k} \chi_k({\det}_{{m_k}}))\\
\cong \ & \Ind_{P_{m_1-1, \ldots, m_k-1}}^{{\Sp}_{2n-2k}} \chi_{\frac{b}{a}} \nu^{\alpha_1} \chi_1 ({\det}_{{m_1-1}}) \otimes \cdots \otimes \nu^{\alpha_k} \chi_k({\det}_{{m_k-1}}),
\end{split}
\end{align}
where $\chi_{\frac{b}{a}}$ is a quadratic character of $F_v^*$ defined by the
Hilbert symbol:
$\chi_{\frac{b}{a}}(x)=(\frac{b}{a}, x)$.
\end{enumerate}
\end{prop}

\begin{proof}
The proof is the same as Theorem 6.3 of \cite{GRS11}.
The key calculation is reduced to that in Proposition 6.6 of \cite{GRS11}.
Explicitly, by \cite[Page 17]{Kud96},
$$
\gamma_{\psi^{a}} \gamma_{\psi^{-b}}
=  \gamma_{\psi^{a}} \gamma_{\psi^{-a}}\chi_{\frac{b}{a}}
= \chi_{\frac{b}{a}}.
$$
\end{proof}

The following proposition can be easily read out from the
Theorem 6.1, Proposition
6.7 and Theorem 6.3 of \cite{GRS11}.

\begin{prop}\label{prop2}
Let $\chi_i$, $1 \leq i \leq r$, be characters of $F_v^*$, and $a \in F_v^*$. Then
\begin{align}\label{prop2equ1}
\begin{split}
 & FJ_{\psi_{k-1}^{a}} (\Ind_{P_{m_1, \ldots, m_k}}^{\Sp_{2n}} \nu^{\alpha_1} \chi_1 ({\det}_{{m_1}}) \otimes \cdots \otimes \nu^{\alpha_k} \chi_k({\det}_{{m_k}}) \otimes 1_{\Sp_{2m}})\\
\cong \ & \Ind_{P_{m_1-1, \ldots, m_k-1}}^{\wt{\Sp}_{2n-2k}} \mu_{\psi^{-a}} \nu^{\alpha_1} \chi_1 ({\det}_{{m_1-1}}) \otimes \cdots \otimes \nu^{\alpha_k} \chi_k({\det}_{{m_k-1}}) \\
& \otimes (1_{\Sp_{2m}} \otimes \omega_{\psi^{-a}}),
\end{split}
\end{align}
where the term $\nu^{\alpha_i} \chi_i ({\det}_{{m_i-1}})$ will be omitted if $m_i=1$, for $1 \leq i \leq k$.
\end{prop}

Before we state the main result of this section, we need to recall the following definition.

\begin{defn}\label{def2}
Given any patition $\ul{q} =[q_1 q_2 \cdots q_r]$ for $\frak{so}_{2n+1}(\BC)$ with $q_1 \geq q_2 \geq \cdots \geq q_r > 0$,
whose even parts occurring with even multiplicity. Let $\ul{q}^- =[q_1 q_2 \cdots q_{r-1} (q_r-1)]$. Then the Barbasch-Vogan duality
$\eta_{\frak{so}_{2n+1}(\BC), \frak{sp}_{2n}(\BC)}$, following \cite[Definition A1]{BV85} and \cite[Section 3.5]{Ac03}, is defined by
$$
\eta_{\frak{so}_{2n+1}(\BC), \frak{sp}_{2n}(\BC)}
(\ul{q}) := ((\ul{q}^-)_{\Sp_{2n}})^t,
$$
where $(\ul{q}^-)_{\Sp_{2n}}$ is the $\Sp_{2n}$-collapse of $\ul{q}^-$.
\end{defn}

In this section, we prove the following theorem.

\begin{thm}\label{thmub2}
Let $\pi$ be an irreducible unitary automorphic representation of $\Sp_{2n}(\BA)$, having, at one unramified local place $v$,
a strongly negative unramified component
$\sigma_{sn,v}$ which is of \textbf{Type I} as in \ref{typeI}. Then, for any symplectic partition $\ul{p}$ of $2n$ with
$$\ul{p} > \eta_{\frak{so}_{2n+1}(\BC), \frak{sp}_{2n}(\BC)}([\prod_{i=1}^l (2m_i+1)]),$$
$\pi$ has no non-vanishing Fourier coefficients attached to $\ul{p}$, in particular, $\ul{p} \notin \mathfrak{p}^m(\pi)$.
\end{thm}

\begin{proof}
By Definition \ref{def2},
$$\eta_{\frak{so}_{2n+1}(\BC), \frak{sp}_{2n}(\BC)}([\prod_{i=1}^l (2m_i+1)]) = [(\prod_{i=2}^l (2m_i+1))_{\Sp} (2m_1)]^t.$$

We prove by induction on $l$. When $l=1$, it is easy to see that $\sigma_{sn,v}$ is the trivial representation, which implies that
for any symplectic partition
$$\ul{p} > \eta_{\frak{so}_{2n+1}(\BC), \frak{sp}_{2n}(\BC)}([(2m_1+1)]) = [(2m_1)]^t=[1^{2m_1}],$$
and a datum $\ul{a}$, the twisted Jacquet module $J_{V_{\ul{p},2}, \psi_{\ul{p},\ul{a}}}(\sigma_{sn,v})$ vanishes identically.
Therefore, $\pi$ has no non-vanishing Fourier coefficients attached to such $\ul{p}$.
We assume that the theorem is true for any $l' < l$.

By assumption, $\sigma_{sn,v}$ is the unique strongly negative unramified constituent of the following induced representation
\begin{align}\label{sec4equ9}
\begin{split}
\rho := \ & \nu^{\frac{m_{l-1}-m_l}{2}} 1_{{\det}_{m_{l-1}+m_l+1}}
\times
\nu^{\frac{m_{l-3}-m_{l-2}}{2}} 1_{{\det}_{m_{l-3}+m_{l-2}+1}}\\
& \times \cdots \times
\nu^{\frac{m_{2}-m_3}{2}} 1_{{\det}_{m_{2}+m_3+1}}
\rtimes 1_{\Sp_{2m_1}}.
\end{split}
\end{align}
And
\begin{equation*}
{\Jord}(\sigma_{sn,v})=\{(1_{\GL_1}, 2m_1+1), (1_{\GL_1}, 2m_2+1),
\ldots, (1_{\GL_1}, 2m_l+1)\},
\end{equation*}
with $2m_1+1 < 2m_2+1 < \cdots < 2m_l+1$.
Since $l$ is odd, write $l=2s+1$.

By Proposition \ref{prop2},
\begin{align}\label{sec4equ14}
\begin{split}
\rho_1
:= \ &  FJ_{\psi_{s-1}^{1}} (\rho)\\
= & \mu_{\psi^{-1}} \nu^{\frac{m_{2s}-m_{2s+1}}{2}} 1_{{\det}_{m_{2s}+m_{2s+1}}}
\times
\nu^{\frac{m_{2s-2}-m_{2s-1}}{2}} 1_{{\det}_{m_{2s-2}+m_{2s-1}}}\\
& \times \cdots \times
\nu^{\frac{m_{2}-m_3}{2}} 1_{{\det}_{m_{2}+m_3}}
\rtimes (1_{\wt{\Sp}_{2m_1}}\otimes \omega_{\psi^{-1}}).
\end{split}
\end{align}

Note that $m_{2s}+m_{2s+1}+1 > m_{2s-2}+m_{2s-1}+1 > \cdots > m_{2}+m_3+1 > 3$. By Lemma \ref{lemv1}, $J_{\psi_{r-1}^{\alpha}}(\rho) \equiv 0$
for any $r \geq s+1$ and any $\alpha \in F^*/(F^*)^2$;
and $J_{\psi_{(2r+1)^2}}(\rho) \equiv 0$ for any $r \geq s$ if $m_1=0$, and for any $r \geq s+1$ if $m_1 >0$.
From Theorem 6.3 of \cite{GRS11}, we can see that $J_{\psi_{r-1}^{\alpha}}(\rho) \equiv 0$ if and only if $FJ_{\psi_{r-1}^{\alpha}}(\rho) \equiv 0$. Therefore, $[(2s)1^{2n-2s}]$ is the maximal partition of the type $[(2r)1^{2n-2r}]$  with respect to which $\rho$ can have a nonzero Fourier-Jacobi module, in this single step.

By \cite[Example 5.4, Page 52]{Kud96},
the unique unramified component of $\rho_1$ in
\eqref{sec4equ14} is the same as the unique unramified component of the following induced representation:
\begin{align}\label{sec4equ15}
\begin{split}
\rho_1'
:= \ & \mu_{\psi^{-1}} \nu^{\frac{m_{2s}-m_{2s+1}}{2}} 1_{{\det}_{m_{2s}+m_{2s+1}}}
\times
\nu^{\frac{m_{2s-2}-m_{2s-1}}{2}} 1_{{\det}_{m_{2s-2}+m_{2s-1}}}\\
& \times \cdots \times
\nu^{\frac{m_{2}-m_3}{2}} 1_{{\det}_{m_{2}+m_3}}
\times \nu^{\frac{-m_1}{2}} 1_{{\det}_{m_1}}
\rtimes 1_{\wt{\Sp}_{0}}.
\end{split}
\end{align}


By Part (2) of Proposition \ref{prop1},
\begin{align}\label{sec4equ16}
\begin{split}
\rho_2
:= \ &  FJ_{\psi_{s}^{-1}} (\rho_1')\\
= & \nu^{\frac{m_{2s}-m_{2s+1}}{2}} 1_{{\det}_{m_{2s}+m_{2s+1}-1}}
\times
\nu^{\frac{m_{2s-2}-m_{2s-1}}{2}} 1_{{\det}_{m_{2s-2}+m_{2s-1}-1}}\\
& \times \cdots \times
\nu^{\frac{m_{2}-m_3}{2}} 1_{{\det}_{m_{2}+m_3-1}}
\times \nu^{\frac{-m_1}{2}} 1_{{\det}_{m_1-1}}
\rtimes 1_{{\Sp}_{0}},
\end{split}
\end{align}
whose irreducible unramified constituent is the same as that of the following induced representation:
\begin{align}\label{sec4equ17}
\begin{split}
\rho_2'
:= \ & \nu^{\frac{m_{2s}-m_{2s+1}}{2}} 1_{{\det}_{m_{2s}+m_{2s+1}-1}}
\times
\nu^{\frac{m_{2s-2}-m_{2s-1}}{2}} 1_{{\det}_{m_{2s-2}+m_{2s-1}-1}}\\
& \times \cdots \times
\nu^{\frac{m_{2}-m_3}{2}} 1_{{\det}_{m_{2}+m_3-1}}
\rtimes 1_{{\Sp}_{2m_1-2}}.
\end{split}
\end{align}

Note that $m_{2s}+m_{2s+1} > m_{2s-2}+m_{2s-1} > \cdots > m_{2}+m_3 > 2$.
Similarly as above, by Lemma \ref{lemv2}, $J_{\psi_{r-1}^{\alpha}}(\rho_1') \equiv 0$ for any $r \geq s+2$ and any $\alpha \in F^*/(F^*)^2$;
and $J_{\psi_{(2r+1)^2}}(\rho_1') \equiv 0$ for any $r \geq s$ if $m_1=1$, and for any $r \geq s+1$ if $m_1 >1$.
Therefore, $[(2s+2)1^{2n-2s-2s-2}]$ is also the maximal partition of the type $[(2r)1^{2n-2s-2r}]$ with respect to which $\rho_1'$ can have a nonzero Fourier-Jacobi module, in this single step. We need to do this routine checking about the ``maximality" using Lemma \ref{lemv1} or Lemma \ref{lemv2}, every time we apply Proposition \ref{prop1} or Proposition \ref{prop2}. We will omit this part in the following steps.

It is easy to see that we can repeat the above 2-step-procedure $m_1-1$ more times, then we get the following induced representation:
\begin{align}\label{sec4equ18}
\begin{split}
 & \rho_{2m_1}\\
 := \ & \mu_{\psi^{-1}} \nu^{\frac{m_{2s}-m_{2s+1}}{2}} 1_{{\det}_{m_{2s}+m_{2s+1}-2m_1+1}}
\times
\nu^{\frac{m_{2s-2}-m_{2s-1}}{2}} 1_{{\det}_{m_{2s-2}+m_{2s-1}-2m_1+1}}\\
& \times \cdots \times
\nu^{\frac{m_{2}-m_3}{2}} 1_{{\det}_{m_{2}+m_3-2m_1+1}}
\rtimes 1_{{\Sp}_{0}}.
\end{split}
\end{align}

Then, we continue with $\rho_{2m_1}$.
By Part (1) of Proposition \ref{prop1},
\begin{align}\label{sec4equ10}
\begin{split}
 & \rho_{2m_1+1}\\
 := \ & FJ_{\psi_{s-1}^{1}} (\rho_{2m_1})\\
 = \ & \mu_{\psi^{-1}} \nu^{\frac{m_{2s}-m_{2s+1}}{2}} 1_{{\det}_{m_{2s}+m_{2s+1}-2m_1}}
\times
\nu^{\frac{m_{2s-2}-m_{2s-1}}{2}} 1_{{\det}_{m_{2s-2}+m_{2s-1}-2m_1}}\\
& \times \cdots \times
\nu^{\frac{m_{2}-m_3}{2}} 1_{{\det}_{m_{2}+m_3-2m_1}}
\rtimes 1_{\wt{\Sp}_{0}}.
\end{split}
\end{align}

By Part (2) of Proposition \ref{prop1},
\begin{align}\label{sec4equ11}
\begin{split}
 & \rho_{2m_1+2}\\
 := \ & FJ_{\psi_{s-1}^{-1}} (\rho_{2m_1+1})\\
 = \ & \nu^{\frac{m_{2s}-m_{2s+1}}{2}} 1_{{\det}_{m_{2s}+m_{2s+1}-2m_1-1}}
\times
\nu^{\frac{m_{2s-2}-m_{2s-1}}{2}} 1_{{\det}_{m_{2s-2}+m_{2s-1}-2m_1-1}}\\
& \times \cdots \times
\nu^{\frac{m_{2}-m_3}{2}} 1_{{\det}_{m_{2}+m_3-2m_1-1}}
\rtimes 1_{{\Sp}_{0}}.
\end{split}
\end{align}

It is easy to see that we can repeat the 2-step-procedure $m_2-m_1$ more times, then
we get the following induced representation
\begin{align}\label{sec4equ12}
\begin{split}
& \rho_{2m_2+2}\\
:= \ & \nu^{\frac{m_{2s}-m_{2s+1}}{2}} 1_{{\det}_{m_{2s}+m_{2s+1}-1-2m_2}}
\times
\nu^{\frac{m_{2s-2}-m_{2s-1}}{2}} 1_{{\det}_{m_{2s-2}+m_{2s-1}-1-2m_2}}\\
& \times \cdots \times
\nu^{\frac{m_{2}-m_3}{2}} 1_{{\det}_{m_3-m_2-1}}
\rtimes 1_{{\Sp}_{0}},
\end{split}
\end{align}
whose unramified component is the same as that of the following induced representation
\begin{align}\label{sec4equ13}
\begin{split}
\rho' := \ & \nu^{\frac{m_{2s}-m_{2s+1}}{2}} 1_{{\det}_{m_{2s}+m_{2s+1}-1-2m_2}}
\times
\nu^{\frac{m_{2s-2}-m_{2s-1}}{2}} 1_{{\det}_{m_{2s-2}+m_{2s-1}-1-2m_2}}\\
& \times \cdots \times
\nu^{\frac{m_{4}-m_{5}}{2}} 1_{{\det}_{m_{4}+m_{5}-1-2m_2}}
\rtimes 1_{{\Sp}_{2m_3-2m_2-2}}.
\end{split}
\end{align}

By Theorem \ref{thm8},
$\rho'$ has a unique strongly negative unramified constituent $\sigma_{sn}'$, and
\begin{align*}
{\Jord}(\sigma_{sn,v}')
= \ & \{(1_{\GL_1}, 2m_3-2m_2-1), (1_{\GL_1}, 2m_4-2m_2-1),\\
& \ldots, (1_{\GL_1}, 2m_{2s+1}-2m_2-1)\},
\end{align*}
with $2s-1$ Jordan blocks.

Note that in general, the unique unramified component of $\rho_{2i}$, $1 \leq i \leq m_2$, may not be strongly negative.

By induction assumption, for any irreducible unitary automorphic representation $\pi'$ of $\Sp_{2m}(\BA)$ which has the unique strongly negative unramified constituent of $\sigma_{sn,v}'$ as a local component, and for any symplectic partition
$\ul{p}$ of $2m$ with
$$\ul{p} > [(\prod_{i=4}^{2s+1} (2m_i-2m_2-1))_{\Sp}(2m_3-2m_2-2)]^t,$$
$\pi'$ has no non-vanishing Fourier coefficients attached to $\ul{p}$.

From the above discussion, we have the following composite partition
\begin{align}\label{sec4equ23}
\begin{split}
& [(2s)1^{2n-2s}]\circ [(2s+2)1^{2n-2(2s+1)}] \circ \cdots \\ \circ \ & [(2s)1^{2n-2m_1(2s+1)+2s+2}]
\circ  [(2s+2)1^{2n-2m_1(2s+1)}]\\
\circ  \ & [(2s)1^{2n-2m_1(2s+1)-2s}]\circ \cdots \circ [(2s)1^{2n-(2m_2+2)2s-2m_1}]\\
\circ \ & [(\prod_{i=4}^{2s+1} (2m_i-2m_2-1))_{\Sp}(2m_3-2m_2-2)]^t,
\end{split}
\end{align}
which may provide non-vanishing Fourier coefficients for $\pi$,
where there are total $m_1$ copies of the pair
$(2s,2s+2)$ in the first two rows, and there are $2m_2+2-2m_1$ copies of
$(2s)$ in the third row.

By Proposition 3.2 of \cite{JL15a},
the composite partition in \eqref{sec4equ23} provides non-vanishing Fourier coefficients for $\pi$ if and only if
the following composite partition
provides non-vanishing Fourier coefficients for $\pi$
\begin{align}\label{sec4equ22}
\begin{split}
& [(2s+1)^21^{2n-2(2s+1)}]\circ \cdots \circ [(2s+1)^21^{2n-2m_1(2s+1)}]\\
\circ \ & [(2s)1^{2n-2m_1(2s+1)-2s}]\circ \cdots \circ [(2s)1^{2n-(2m_2+2)2s-2m_1}]\\
\circ \ & [(\prod_{i=4}^{2s+1} (2m_i-2m_2-1))_{\Sp}(2m_3-2m_2-2)]^t.
\end{split}
\end{align}

By the recipe in Theorem 6.3.8 of \cite{CM93} (see the beginning of the current section), it is easy to see that
\begin{align}\label{sec4equ24}
\begin{split}
& [(\prod_{i=4}^{2s+1} (2m_i-2m_2-1))_{Sp}(2m_3-2m_2-2)]^t\\
= \ & [(2m_{2s+1}-2m_2-2)(2m_{2s}-2m_2) \cdots (2m_{5}-2m_2-2)(2m_4-2m_2)\\
& \cdot (2m_3-2m_2-2)]^t\\
= \ & 1^{2m_{2s+1}-2m_2-2} + 1^{2m_{2s}-2m_2} + \cdots + 1^{2m_{5}-2m_2-2} + 1^{2m_4-2m_2} \\
& + 1^{2m_3-2m_2-2}\\
= \ & [(2s-1)^{2m_3-2m_2-2}(2s-2)^{2m_4+2-2m_3}(2s-3)^{2m_5-2m_4-2} \cdots \\
& (2)^{2m_{2s}+2-2m_{2s-1}} 1^{2m_{2s+1}-2m_{2s}-2}].
\end{split}
\end{align}
Therefore, by \cite[Lemma 3.1]{JL15a} or  \cite[Lemma 2.6]{GRS03}, and \cite[Proposition 3.3]{JL15a},
if the composite partition in \eqref{sec4equ22} provides non-vanishing Fourier coefficients for $\pi$, then so is
the following partition
\begin{equation}\label{sec4equ21}
[(2s+1)^{2m_1}(2s)^{2m_2+2-2m_1}((\prod_{i=4}^{2s+1} (2m_i-2m_2-1))_{Sp}(2m_3-2m_2-2))^t].
\end{equation}


Given a symplectic partition $\ul{p}=[p_1^{e_1}p_2^{e_2} \cdots p_r^{e_r}]$ of $2n$, that is, $e_i=1$ if $p_i$ is even, $e_i=2$ if $p_i$ is odd. Assume that $\ul{p}$ is bigger than the partition in \eqref{sec4equ21}. Write the symplectic partition in \eqref{sec4equ21} as $\ul{q} = [q_1^{e_1}q_2^{e_2} \cdots q_t^{e_t}]$. Assume that $1 \leq i_0 \leq r$ is the unique index such that
$p_i = q_i$ for $1 \leq i < i_0$, and $p_{i_0} > q_{i_0}$. If $i_0=1$, then by Lemma \ref{lemv1}, $\rho$ in \eqref{sec4equ9} has no nonzero twisted Jacquet module attached to the partition $[p_1^{e_1} 1^{2n-p_1e_1}]$.
Therefore, $\pi$ has no nonzero Fourier coefficients attached to the partition $[p_1^{e_1} 1^{2n-p_1e_1}]$. By Lemma 3.1 of \cite{JL15a} or Lemma 2.6 of \cite{GRS03}, and Proposition 3.3 of \cite{JL15a}, $\pi$ has no nonzero Fourier coefficients attached to the partition $\ul{p}$.

Next, we may assume that $i_0>1$. For $p_1=q_1=2s$, which means that $m_1=0$, then we take the Fourier-Jacobi module $FJ_{\psi_{s-1}^{1}}$ as in \eqref{sec4equ14} for $\rho$ in \eqref{sec4equ9}. If $p_1=q_1=2s+1$, which implies that $m_1>0$, then by Proposition 3.2 of \cite{JL15a}, to consider the Fourier coefficients attached to the partition $[(2s+1)^21^{2n-2(2s+1)}]$, it  suffices to consider the composite partition $[(2s)1^{2n-2s}]\circ[(2s+2)1^{2n-2(2s+1)}]$.
Take the Fourier-Jacobi modules $FJ_{\psi_{s-1}^{1}}$ and $FJ_{\psi_{s}^{-1}}$ consequently as in \eqref{sec4equ14} and \eqref{sec4equ16} for $\rho$ in \eqref{sec4equ9}.
After repeating the above procedure for $p_i=q_i$, $2 \leq i < i_0$, we come to a similar situation as in the case of $i_0=1$. Using a similar argument as in the case of $i_0=1$, applying first  Lemma \ref{lemv1} if having taken even times of Fourier-Jacobi modules or Lemma \ref{lemv2} otherwise,
then Lemma 3.1 of \cite{JL15a} or Lemma 2.6 of \cite{GRS03}, and Proposition 3.3 of \cite{JL15a}
we can conclude that $\pi$ also has no nonzero Fourier coefficients attached to the partition $\ul{p}$.

Hence, for any symplectic partition $\ul{p}$ which is bigger than the partition in \eqref{sec4equ21}, $\pi$ has no nonzero Fourier coefficients attached to the partition $\ul{p}$.
Therefore, it remains to show that the partition in \eqref{sec4equ21} is exactly equal to
$\eta_{\frak{so}_{2n+1}(\BC), \frak{sp}_{2n}(\BC)}([\prod_{i=1}^l (2m_i+1)])$.

By a similar calculation as in \eqref{sec4equ24}, $[(\prod_{i=2}^l (2m_i+1))_{Sp} (2m_1)]^t$ equals
\begin{align*}
& [(2s+1)^{2m_1}(2s)^{2m_2+2-2m_1}(2s-1)^{2m_3-2m_2-2}(2s-2)^{2m_4+2-2m_3}\\
& (2s-3)^{2m_5-2m_4-2} \cdots
 (2)^{2m_{2s}+2-2m_{2s-1}} 1^{2m_{2s+1}-2m_{2s}-2}].
\end{align*}
Therefore, the partition
$$
[(2s+1)^{2m_1}(2s)^{2m_2+2-2m_1}((\prod_{i=4}^{2s+1} (2m_i-2m_2-1))_{Sp}(2m_3-2m_2-2))^t]
$$
is equal to
$$
[(\prod_{i=2}^l (2m_i+1))_{Sp} (2m_1)]^t = \eta_{\frak{so}_{2n+1}(\BC), \frak{sp}_{2n}(\BC)}([\prod_{i=1}^l (2m_i+1)]).
$$
Hence, for any symplectic partition
$\ul{p}$ of $2n$ with
$$\ul{p} > \eta_{\frak{so}_{2n+1}(\BC), \frak{sp}_{2n}(\BC)}([\prod_{i=1}^l (2m_i+1)]),$$
$\pi$ has no non-vanishing Fourier coefficients attached to $\ul{p}$,
in particular, $\ul{p} \notin \mathfrak{p}^m(\pi)$.
This completes the proof of the theorem.
\end{proof}

Applying Theorem \ref{thmub2}, we can easily obtain the following analogue result.

\begin{thm}\label{thmub3}
Let $\pi$ be an irreducible unitary automorphic representation of $\Sp_{2n}(\BA)$, having, at one unramified local place $v$,
a strongly negative unramified component $\sigma_{sn,v}$ which is of \textbf{Type II} as in \ref{typeII}.
Then, for any symplectic partition $\ul{p}$ of $2n$ with
$$\ul{p} > \eta_{\frak{so}_{2n+1}(\BC), \frak{sp}_{2n}(\BC)}([\prod_{i=1}^k (2n_i+1)(1)]),$$
$\pi$ has no non-vanishing Fourier coefficients attached to $\ul{p}$,
in particular, $\ul{p} \notin \mathfrak{p}^m(\pi)$.
\end{thm}

\begin{proof}
By assumption, $\sigma_{sn,v}$ is the unique strongly negative unramified constituent of the following induced representation
\begin{align}\label{sec4equ19}
\begin{split}
\rho :=  \ & \nu^{\frac{n_{k-1}-n_k}{2}} \lambda_0({\det}_{n_{k-1}+n_k+1})
\times
\nu^{\frac{n_{k-3}-n_{k-2}}{2}} \lambda_0({\det}_{n_{k-3}+n_{k-2}+1})\\
& \times \cdots \times
\nu^{\frac{n_{1}-n_2}{2}} \lambda_0({\det}_{n_{1}+n_2+1})
\rtimes 1_{\Sp_{0}}.
\end{split}
\end{align}
And
\begin{equation*}
{\Jord}(\sigma_{sn,v})=\{(\lambda_0, 2n_1+1), (\lambda_0, 2n_2+1),
\ldots, (\lambda_0, 2n_k+1), (1_{\GL_1},1)\},
\end{equation*}
with $2n_1+1 < 2n_2+1 < \cdots < 2n_k+1$ and $k$ being even.

It is easy to see that $\rho$ can be written as $\lambda_0 \rho'$, where:
\begin{align}\label{sec4equ20}
\begin{split}
\rho' :=  \ & \nu^{\frac{n_{k-1}-n_k}{2}} 1_{{\det}_{n_{k-1}+n_k+1}}
\times
\nu^{\frac{n_{k-3}-n_{k-2}}{2}} 1_{{\det}_{n_{k-3}+n_{k-2}+1}}\\
& \times \cdots \times
\nu^{\frac{n_{1}-n_2}{2}} 1_{{\det}_{n_{1}+n_2+1}}
\rtimes 1_{\Sp_{0}}.
\end{split}
\end{align}

By Theorem \ref{thm8}, $\rho'$ also has a unique strongly negative unramified component $\sigma_{sn,v}'$ with
\begin{equation*}
{\Jord}(\sigma_{sn,v}')=\{(1, 2n_1+1), (1, 2n_2+1),
\ldots, (1, 2n_k+1), (1_{\GL_1},1)\}.
\end{equation*}

Applying the argument in the proof of Theorem \ref{thmub2} to $\rho'$, we can easily see that for any symplectic partition $\ul{p}$ of $2n$ with
$$\ul{p} > \eta_{\frak{so}_{2n+1}(\BC), \frak{sp}_{2n}(\BC)}([\prod_{i=1}^k (2n_i+1)(1)]),$$
$\pi$ has no non-vanishing Fourier coefficients attached to $\ul{p}$,
in particular, $\ul{p} \notin \mathfrak{p}^m(\pi)$.
This proves the theorem.
\end{proof}

\begin{rmk}\label{rmk7}
If an irreducible unitary automorphic representation $\pi$ of $\Sp_{2n}(\BA)$ has as a unramified local component a strongly negative unramified component $\sigma_{sn}$ that is neither of \textbf{Type I} nor \textbf{Type II}, that is, two characters
$\lambda_0$ and $1_{\GL_1}$ are mixed, then the above computation will get more complicated. We omit the detail here.
\end{rmk}


\section{Vanishing of Certain Fourier Coefficients: General Case}

In this section, we continue to characterize the vanishing property of Fourier coefficients for certain irreducible automorphic representations, based on local unramified information. We
prove the following theorem, which is a generalization of Theorem \ref{thmub2}.

\begin{thm}\label{thmub1}
Let $\pi$ be an irreducible unitary automorphic representation of $\Sp_{2n}(\BA)$ which has, at one unramified local place $v$
an unramified component $\sigma_v$ of \textbf{Type III} as in \ref{typeIII}. Then the following hold.
\begin{enumerate}
\item For any symplectic partition $\ul{p}$ of $2n$ with
$$\ul{p} > \ul{p}_1 := ([(\prod_{j=1}^t n_j^2) (\prod_{(\chi,m,\alpha) \in \textbf{e}} m^2) (\prod_{i=2}^l (2m_i+1))_{\Sp}(2m_1)]^t)_{\Sp},$$
$\pi$ has no non-vanishing Fourier coefficients attached to $\ul{p}$, in particular, $\ul{p} \notin \mathfrak{p}^m(\pi)$.
\item The partition $\ul{p}_1$ has the property that
$$\ul{p}_1 = \eta_{\frak{so}_{2n+1}(\BC), \frak{sp}_{2n}(\BC)}([(\prod_{j=1}^t n_j^2) (\prod_{(\chi,m,\alpha) \in \textbf{e}} m^2) (\prod_{i=1}^l (2m_i+1))]).$$
\end{enumerate}
\end{thm}

\begin{proof}
\textbf{Proof of Part (1).}
We prove by induction on $n$. When $n=1$, then $k=0$, and either $m_1=0$ or $m_1=1$.
If $m_1=0$, then by Part (1) of Proposition \ref{prop1},
$\ul{p}_1=[1^2]^t=[2]$.
If $m_0=1$, then $\ul{p}_1=[2]^t=[1^2]$.
Therefore, Part (1) is true for $n=1$.
We assume that the result is true for any $n' < n$.

Since $l$ is odd, we assume that $l=2s+1$. By the assumption of the theorem, $\sigma_v$ is the unique unramified constituent of the following induced representation:
\begin{align*}
\begin{split}
\rho
:=  \ & \times_{(\chi, m, \alpha) \in \textbf{e}}
v^{\alpha} \chi({\det}_m) \times \times_{j=1}^t \chi_j ({\det}_{{n_j}})\\
& \times
\nu^{\frac{m_{2s}-m_{2s+1}}{2}} 1_{{\det}_{m_{2s}+m_{2s+1}+1}}
\times
\nu^{\frac{m_{2s-2}-m_{2s-1}}{2}} 1_{{\det}_{m_{2s-2}+m_{2s-1}+1}}\\
& \times \cdots \times
\nu^{\frac{m_{2}-m_3}{2}} 1_{{\det}_{m_{2}+m_3+1}}
\rtimes 1_{\Sp_{2m_1}}.
\end{split}
\end{align*}

We assume that
$$[(\prod_{j=1}^t n_j^2) (\prod_{(\chi,m,\alpha) \in \textbf{e}} m^2)]^t = [(2q_1)(2q_2)\cdots(2q_r)],$$
where $2q_1 \geq 2q_2 \geq \cdots \geq 2q_r$. If $r < 2m_2+2$, then we let $q_{r+1} = \cdots q_{2m_2+2} = 0$.
By Proposition \ref{prop2},
\begin{align*}
\begin{split}
\rho_1
:=  \ & FJ_{\psi_{q_1+s-1}^1}(\rho)\\
=  \ & \mu_{\psi^{-1}} \times_{(\chi, m, \alpha) \in \textbf{e}}
v^{\alpha} \chi({\det}_{m-1}) \times \times_{j=1}^t \chi_j ({\det}_{{n_j-1}})\\
& \times
\nu^{\frac{m_{2s}-m_{2s+1}}{2}} 1_{{\det}_{m_{2s}+m_{2s+1}}}
\times
\nu^{\frac{m_{2s-2}-m_{2s-1}}{2}} 1_{{\det}_{m_{2s-2}+m_{2s-1}}}\\
& \times \cdots \times
\nu^{\frac{m_{2}-m_3}{2}} 1_{{\det}_{m_{2}+m_3}}
\rtimes (1_{\Sp_{2m_1}} \otimes \omega_{\psi^{-1}}),
\end{split}
\end{align*}
where we make a convention that a term $v^{a} \chi({\det}_{b-1})$ will be omitted from the induced representation if $b-1 \leq 0$. From now on, we will follow this convention.

Similarly as in the proof of Theorem \ref{thmub2},
by Lemma \ref{lemv2}, $J_{\psi_{r-1}^{\alpha}}(\rho) \equiv 0$, for any $r \geq q_1 + s +1$ and any $\alpha \in F^*/(F^*)^2$,
and $J_{\psi_{(2r+1)^2}}(\rho) \equiv 0$, for any $r \geq q_1 + s$ if $m_1=0$, or, if $m_1 > 0$ and some $m$ or $n_j$ is $1$;
 for any $r \geq q_1 + s +1$ if $m_1 >0$ and all $m, n_j$'s are bigger than 1.
Note that all $m, n_j$'s are bigger than 1 if and only if $2q_1 = 2q_2$.
Therefore, $J_{\psi_{(2r+1)^2}}(\rho) \equiv 0$, for any $r \geq q_1 + s$ if $m_1=0$, or, if $2q_1 \geq 2q_2+2$;
 for any $r \geq q_1 + s +1$ if $m_1 >0$ and $2q_1 = 2q_2$.

Therefore, $[(2q_1+2s)1^{2n-2q_1-2s}]$ is the maximal partition of the type $[(2r)1^{2n-2r}]$ with respect to which $\rho$ can have a nonzero Fourier-Jacobi module, in this single step. We need to do this routine
by checking about the ``maximality" using Lemma \ref{lemv1} or Lemma \ref{lemv2}, every time we apply Proposition \ref{prop1} or Proposition \ref{prop2}. We will omit this part in the following steps.

By \cite{Kud96},
the unique unramified component of
$\rho_1$ is the same as the unique unramified component of the following induced representation:
\begin{align*}
\rho_1'
:=  \ & \mu_{\psi^{-1}} \times_{(\chi, m, \alpha) \in \textbf{e}}
v^{\alpha} \chi({\det}_{m-1}) \times \times_{j=1}^t \chi_j ({\det}_{{n_j-1}})\\
& \times
\nu^{\frac{m_{2s}-m_{2s+1}}{2}} 1_{{\det}_{m_{2s}+m_{2s+1}}}
\times
\nu^{\frac{m_{2s-2}-m_{2s-1}}{2}} 1_{{\det}_{m_{2s-2}+m_{2s-1}}}\\
& \times \cdots \times
\nu^{\frac{m_{2}-m_3}{2}} 1_{{\det}_{m_{2}+m_3}}
\times \nu^{\frac{-m_1}{2}} 1_{{\det}_{m_1}}
\rtimes 1_{\wt{\Sp}_{0}}.
\end{align*}
By Part (2) of Proposition \ref{prop2},
\begin{align*}
\begin{split}
\rho_2
:=  \ & FJ_{\psi_{q_2+s}^{-1}}(\rho_1')\\
=  \ & \times_{(\chi, m, \alpha) \in \textbf{e}}
v^{\alpha} \chi({\det}_{m-2}) \times \times_{j=1}^t \chi_j ({\det}_{{n_j-2}})\\
& \times
\nu^{\frac{m_{2s}-m_{2s+1}}{2}} 1_{{\det}_{m_{2s}+m_{2s+1}-1}}
\times
\nu^{\frac{m_{2s-2}-m_{2s-1}}{2}} 1_{{\det}_{m_{2s-2}+m_{2s-1}-1}}\\
& \times \cdots \times
\nu^{\frac{m_{2}-m_3}{2}} 1_{{\det}_{m_{2}+m_3-1}} \times
\nu^{\frac{-m_1}{2}} 1_{{\det}_{m_1-1}}
\rtimes 1_{\Sp_{0}}.
\end{split}
\end{align*}
It is easy to see that we can repeat the above 2-step-procedure $m_1-1$ more times, then we get the following induced representation:
\begin{align*}
& \rho_{2m_1}\\
:=  \ & FJ_{\psi_{q_{2m_1}+s}^{-1}}(\rho_{2m_1-1}')\\
=  \ & \times_{(\chi, m, \alpha) \in \textbf{e}}
v^{\alpha} \chi({\det}_{m-2m_1}) \times \times_{j=1}^t \chi_j ({\det}_{{n_j-2m_1}})\\
& \times
\nu^{\frac{m_{2s}-m_{2s+1}}{2}} 1_{{\det}_{m_{2s}+m_{2s+1}-2m_1+1}}
\times
\nu^{\frac{m_{2s-2}-m_{2s-1}}{2}} 1_{{\det}_{m_{2s-2}+m_{2s-1}-2m_1+1}}\\
& \times \cdots \times
\nu^{\frac{m_{2}-m_3}{2}} 1_{{\det}_{m_{2}+m_3-2m_1+1}}
\rtimes 1_{\Sp_{0}}.
\end{align*}
Then, we continue with $\rho_{2m_1}$.
By Part (1) of Proposition \ref{prop1},
\begin{align*}
& \rho_{2m_1+1}\\
:=  \ & FJ_{\psi_{q_{2m_1+1}+s-1}^{1}}(\rho_{2m_1})\\
=  \ & \mu_{\psi^{-1}} \times_{(\chi, m, \alpha) \in \textbf{e}}
v^{\alpha} \chi({\det}_{m-2m_1-1}) \times \times_{j=1}^t \chi_j ({\det}_{{n_j-2m_1-1}})\\
& \times
\nu^{\frac{m_{2s}-m_{2s+1}}{2}} 1_{{\det}_{m_{2s}+m_{2s+1}-2m_1}}
\times
\nu^{\frac{m_{2s-2}-m_{2s-1}}{2}} 1_{{\det}_{m_{2s-2}+m_{2s-1}-2m_1}}\\
& \times \cdots \times
\nu^{\frac{m_{2}-m_3}{2}} 1_{{\det}_{m_{2}+m_3-2m_1}}
\rtimes 1_{\wt{\Sp}_{0}}.
\end{align*}
By Part (2) of Proposition \ref{prop1},
\begin{align*}
& \rho_{2m_1+2}\\
:=  \ & FJ_{\psi_{q_{2m_1+2}+s-1}^{-1}}(\rho_{2m_1+1})\\
=  \ & \times_{(\chi, m, \alpha) \in \textbf{e}}
v^{\alpha} \chi({\det}_{m-2m_1-2}) \times \times_{j=1}^t \chi_j ({\det}_{{n_j-2m_1-2}})\\
& \times
\nu^{\frac{m_{2s}-m_{2s+1}}{2}} 1_{{\det}_{m_{2s}+m_{2s+1}-2m_1-1}}
\times
\nu^{\frac{m_{2s-2}-m_{2s-1}}{2}} 1_{{\det}_{m_{2s-2}+m_{2s-1}-2m_1-1}}\\
& \times \cdots \times
\nu^{\frac{m_{2}-m_3}{2}} 1_{{\det}_{m_{2}+m_3-2m_1-1}}
\rtimes 1_{{\Sp}_{0}}.
\end{align*}

It is easy to see that after repeating the above 2-step-procedure $m_2-m_1+1$ more times, we get the following induced representation:
\begin{align*}
& \rho_{2m_2+2}\\
:=  \ & FJ_{\psi_{q_{2m_2+2}+s-1}^{-1}}(\rho_{2m_2+1})\\
=  \ & \times_{(\chi, m, \alpha) \in \textbf{e}}
v^{\alpha} \chi({\det}_{m-2m_2-2}) \times \times_{j=1}^t \chi_j ({\det}_{{n_j-2m_2-2}})\\
& \times
\nu^{\frac{m_{2s}-m_{2s+1}}{2}} 1_{{\det}_{m_{2s}+m_{2s+1}-2m_2-1}}
\times
\nu^{\frac{m_{2s-2}-m_{2s-1}}{2}} 1_{{\det}_{m_{2s-2}+m_{2s-1}-2m_2-1}}\\
& \times \cdots \times
\nu^{\frac{m_{2}-m_3}{2}} 1_{{\det}_{m_{2}+m_3-2m_2-1}}
\rtimes 1_{{\Sp}_{0}},
\end{align*}
whose unramified component is the same as that of the following induced representation
\begin{align*}
& \rho'\\
:=  \ & \times_{(\chi, m, \alpha) \in \textbf{e}}
v^{\alpha} \chi({\det}_{m-2m_2-2}) \times \times_{j=1}^t \chi_j ({\det}_{{n_j-2m_2-2}})\\
& \times
\nu^{\frac{m_{2s}-m_{2s+1}}{2}} 1_{{\det}_{m_{2s}+m_{2s+1}-2m_2-1}}
\times
\nu^{\frac{m_{2s-2}-m_{2s-1}}{2}} 1_{{\det}_{m_{2s-2}+m_{2s-1}-2m_2-1}}\\
& \times \cdots \times
\nu^{\frac{m_{4}-m_{5}}{2}} 1_{{\det}_{m_{4}+m_{5}-1-2m_2}}
\rtimes 1_{{\Sp}_{2m_3-2m_2-2}}.
\end{align*}
By Theorem \ref{thm10}, $\rho'$ has a unique unitary unramified representation $\sigma'$ which corresponds to the following set of data:
\begin{align*}
& \{(\chi, m-2m_2-2, \alpha): (\chi, m, \alpha)\in \textbf{e}\} \\
\cup  \ & \{(\chi_j, n_j-2m_2-2):1 \leq j \leq t\}\\
\cup  \ & \{(1_{\GL_1}, 2m_3-2m_2-1), (1_{\GL_1}, 2m_4-2m_2-1),\\
& \ldots, (1_{\GL_1}, 2m_{2s+1}-2m_2-1)\},
\end{align*}
where terms $(\chi, m-2m_2-2, \alpha)$ or $(\chi_j, n_j-2m_2-2)$ will be omitted if $m-2m_2-2 \leq 0$ or $n_j-2m_2-2 \leq 0$.

Note that in general, it is not easy to figure out the exact corresponding data for the unique unramified component of $\rho_{2i}$, $1 \leq i \leq m_2$.

By induction assumption, for any irreducible unitary automorphic representation $\pi'$ of $\Sp_{2m}(\BA)$ which has the unique strongly negative unramified constituent of $\sigma'_v$ as a local component, and
for any symplectic partition $\ul{p}$ of $2m$ with
\begin{align*}
\ul{p} >  \ & ([(\prod_{j=1}^t (n_j-2m_2-2)^2) (\prod_{(\chi,m,\alpha) \in \textbf{e}} (m-2m_2-2)^2) \\
& \cdot (\prod_{i=4}^l (2m_i-2m_2-1))_{Sp}(2m_3-2m_2-2)]^t)_{\Sp},
\end{align*}
$\pi'$ has no non-vanishing Fourier coefficients attached to $\ul{p}$.

From the above discussion, we have the following composite partition
\begin{align}\label{sec5equ9}
\begin{split}
& [(2q_1+2s)1^{2n-2q_1-2s}]\circ [(2q_2+2s+2)1^{2n-\sum_{i=1}^{2}2q_i-2(2s+1)}] \circ \cdots \\
\circ  \ & [(2q_{2m_1-1}+2s)
1^{2n-\sum_{i=1}^{2m_1-1}2q_i-2m_1(2s+1)+2s+2}]\\
\circ  \ & [(2q_{2m_1}+2s+2)
1^{2n-\sum_{i=1}^{2m_1}2q_i-2m_1(2s+1)}]\\
\circ  \ &  [(2q_{2m_1+1}+2s)
1^{2n-\sum_{i=1}^{2m_1+1}2q_i-2m_1(2s+1)-2s}]\circ \cdots \\
\circ  \ & [(2q_{2m_2+2}+2s)
1^{2n-\sum_{i=1}^{2m_2+1}2q_i-(2m_2+2)2s-2m_1}]\\
\circ  \ & ([(\prod_{j=1}^t (n_j-2m_2-2)^2) (\prod_{(\chi,m,\alpha) \in \textbf{e}} (m-2m_2-2)^2) \\
\cdot  \ & (\prod_{i=4}^l (2m_i-2m_2-1))_{Sp}(2m_3-2m_2-2)]^t)_{\Sp},
\end{split}
\end{align}
which may provide non-vanishing Fourier coefficients for $\pi$.

For the partition
$$\ul{p}_1 := ([(\prod_{j=1}^t n_j^2) (\prod_{(\chi,m,\alpha) \in \textbf{e}} m^2) (\prod_{i=2}^l (2m_i+1))_{\Sp}(2m_1)]^t)_{\Sp},$$
since $2m_1+1<2m_2+1<\cdots<2m_{2s+1}+1$,
\begin{align*}
&(\prod_{i=2}^l (2m_i+1))_{\Sp}\\
=  \ & [(2m_{2s+1})(2m_{2s}+2)\cdots(2m_5)(2m_4+2)(2m_3)(2m_2+2)].
\end{align*}
Therefore
\begin{align*}
\ul{p}_2
:=  \ & [(\prod_{j=1}^t n_j^2) (\prod_{(\chi,m,\alpha) \in \textbf{e}} m^2) (\prod_{i=2}^l (2m_i+1))_{\Sp}(2m_1)]^t\\
=  \ & [(\prod_{j=1}^t n_j^2) (\prod_{(\chi,m,\alpha) \in \textbf{e}} m^2)]^t\\
& + [(2m_{2s+1})(2m_{2s}+2)\cdots(2m_5)(2m_4+2)(2m_3)(2m_2+2)(2m_1)]^t.
\end{align*}
By calculating the transpose and the addition, we obtain
\begin{align*}
\ul{p}_2=  \ & [(2q_1)(2q_2)\cdots(2q_r)] +
[1^{2m_{2s+1}}]+[1^{2m_{2s}+2}]\\
& + \cdots + [1^{2m_5}] + [1^{2m_4+2}]
+ [1^{2m_3}] + [1^{2m_2+2}][1^{2m_1}]\\
=  \ & [(2q_1+2s+1)\cdots(2q_{2m_1}+2s+1)\\
& \cdot
(2q_{2m_1+1}+2s)\cdots(2q_{2m_2+2}+2s)(\ul{p}_3)],
\end{align*}
where
\begin{align*}
\ul{p}_3
=  \ & [(2q_{2m_2+3})(2q_{2m_2+4})\cdots(2q_r)] +
[1^{2m_{2s+1}-2m_2-2}]+[1^{2m_{2s}-2m_2}]\\
& + \cdots + [1^{2m_5-2m_2-2}] + [1^{2m_4-2m_2}]
+ [1^{2m_3-2m_2-2}]\\
=  \ & [(\prod_{j=1}^t (n_j-2m_2-2)^2) (\prod_{(\chi,m,\alpha) \in \textbf{e}} (m-2m_2-2)^2) \\
& \cdot (\prod_{i=4}^l (2m_i-2m_2-1))_{\Sp}(2m_3-2m_2-2)]^t.
\end{align*}

By the recipe for symplectic collapse
(see Theorem 6.3.8 of \cite{CM93}, also the beginning of Section 4)
\begin{align}\label{sec5equ11}
\begin{split}
\ul{p}_1
= (\ul{p}_2)_{\Sp}
=  \ & [(2q_1+2s+1)\cdots(2q_{2m_1}+2s+1)\\
& \cdot
(2q_{2m_1+1}+2s)\cdots(2q_{2m_2+2}+2s)(\ul{p}_3)]_{\Sp}\\
=  \ & [((2q_1+2s+1)\cdots(2q_{2m_1}+2s+1)\\
& \cdot
(2q_{2m_1+1}+2s)\cdots(2q_{2m_2+2}+2s))_{\Sp}
(\ul{p}_3)_{\Sp}]\\
=  \ & [((2q_1+2s+1)\cdots(2q_{2m_1}+2s+1))_{\Sp}\\
& \cdot
(2q_{2m_1+1}+2s)\cdots(2q_{2m_2+2}+2s)
(\ul{p}_3)_{\Sp}].
\end{split}
\end{align}

Now, let us come back to the composition partition in \eqref{sec5equ9}.
By Lemma 3.1 of \cite{JL15a} or Lemma 2.6 of \cite{GRS03}, if the composite partition in \eqref{sec5equ9} provides non-vanishing Fourier coefficients for $\pi$, then so is the following composite partition
\begin{align}\label{sec5equ10}
\begin{split}
& [(2q_1+2s)1^{2n-2q_1-2s}]\circ [(2q_2+2s+2)1^{2n-\sum_{i=1}^{2}2q_i-2(2s+1)}] \circ \cdots \\
\circ  \ & [(2q_{2m_1-1}+2s)
1^{2n-\sum_{i=1}^{2m_1-1}2q_i-2m_1(2s+1)+2s+2}]\\
\circ  \ & [(2q_{2m_1}+2s+2)
1^{2n-\sum_{i=1}^{2m_1}2q_i-2m_1(2s+1)}]\\
\circ  \ &  [(2q_{2m_1+1}+2s) \cdots  (2q_{2m_2+2}+2s)\\
\cdot  \ & (((\prod_{j=1}^t (n_j-2m_2-2)^2) (\prod_{(\chi,m,\alpha) \in \textbf{e}} (m-2m_2-2)^2) \\
\cdot  \ & (\prod_{i=4}^l (2m_i-2m_2-1))_{Sp}(2m_3-2m_2-2))^t)_{\Sp}],
\end{split}
\end{align}
which can be expressed as the following partition
\begin{align}
\begin{split}
& [(2q_1+2s)1^{2n-2q_1-2s}]\circ [(2q_2+2s+2)1^{2n-\sum_{i=1}^{2}2q_i-2(2s+1)}] \circ \cdots \\
\circ  \ & [(2q_{2m_1-1}+2s)
1^{2n-\sum_{i=1}^{2m_1-1}2q_i-2m_1(2s+1)+2s+2}]\\
\circ  \ & [(2q_{2m_1}+2s+2)
1^{2n-\sum_{i=1}^{2m_1}2q_i-2m_1(2s+1)}]\\
\circ  \ & [(2q_{2m_1+1}+2s) \cdots  (2q_{2m_2+2}+2s) (\ul{p}_3)_{\Sp})].
\end{split}
\end{align}

Comparing the partitions in \eqref{sec5equ11} and \eqref{sec5equ10},
and applying Lemma 3.1 and Proposition 3.3 of \cite{JL15a} repeatedly,
we want to show that
if the following composite partition
\begin{align*}
& [(2q_1+2s)1^{2n-2q_1-2s}]\circ [(2q_2+2s+2)1^{2n-\sum_{i=1}^{2}2q_i-2(2s+1)}] \circ \cdots \\
\circ  \ & [(2q_{2m_1-1}+2s)
1^{2n-\sum_{i=1}^{2m_1-1}2q_i-2m_1(2s+1)+2s+2}]\\
\circ  \ & [(2q_{2m_1}+2s+2)
1^{2n-\sum_{i=1}^{2m_1}2q_i-2m_1(2s+1)}]
\end{align*}
provides non-vanishing Fourier coefficients for $\pi$,
then
so is the partition $[(2q_1+2s+1)\cdots(2q_{2m_1}+2s+1)
1^{2n-\sum_{i=1}^{2m_1}2q_i-2m_1(2s+1)}]_{\Sp}$.

We consider each pair $(2q_{2i-1}+2s, 2q_{2i}+2s+2)$, for $1 \leq i \leq m_1$. When $2q_{2i-1}=2q_{2i}$, by Proposition 3.2 of \cite{JL15a},
the composite partition $[(2q_{2i}+2s)1^{2d_i-2q_{2i}-2s}]\circ [(2q_{2i}+2s+2)1^{2d_i-2q_{2i}-2(2s+1)}]$ provides non-vanishing Fourier coefficients for an irreducible automorphic representation $\tau_i$ of $\Sp_{2d_i}(\BA)$, where $2d_i = 2n-\sum_{j=1}^{2i-2} 2q_j-(2i-2)(2s+1)$,
if and only if the partition $[(2q_{2i}+2s+1)^21^{2d_i-2(2q_{2i}+2s+1)}]$
provides non-vanishing Fourier coefficients for $\tau_i$.
When $2q_{2i-1} \geq 2q_{2i}+2$,
by Lemma 3.1 of \cite{JL15a},
if
$$[(2q_{2i-1}+2s)1^{2d_i-2q_{2i-1}-2s}]\circ [(2q_{2i}+2s+2)1^{2d_i-q_{2i-1} -q_{2i}-2(2s+1)}]$$ provides non-vanishing Fourier coefficients for $\tau_i$,
then so is
$$
[(2q_{2i-1}+2s)(2q_{2i}+2s+2)1^{2d_i-q_{2i-1} -q_{2i}-2(2s+1)}].
$$

By the recipe for symplectic collapse
(see Theorem 6.3.8 of \cite{CM93}, also the beginning of Section 4), after considering all the pairs $(2q_{2i-1}+2s, 2q_{2i}+2s+2)$,
$1 \leq i \leq m_1$, replacing
$[(2q_{2i-1}+2s)1^{2d_i-2q_{2i-1}-2s}]\circ [(2q_{2i}+2s+2)1^{2d_i-q_{2i-1} -q_{2i}-2(2s+1)}]$ by $[(2q_{2i}+2s+1)^21^{2d_i-2(2q_{2i}+2s+1)}]$ if $2q_{2i-1}=2q_{2i}$,
by $[(2q_{2i-1}+2s)(2q_{2i}+2s+2)1^{2d_i-q_{2i-1} -q_{2i}-2(2s+1)}]$ if $2q_{2i-1} \geq 2q_{2i}+2$, and applying Lemma 3.1 and Proposition 3.3 of \cite{JL15a} repeatedly, we will get a partition which is exactly 
$$
[(2q_1+2s+1)\cdots(2q_{2m_1}+2s+1)1^{2n-\sum_{i=1}^{2m_1}2q_i-2m_1(2s+1)}]_{\Sp},
$$
providing non-vanishing Fourier coefficients for $\pi$.
Therefore, from the partition in \eqref{sec5equ10}, we get exactly the partition $\ul{p}_1$.


Using a similar argument as in the proof of Theorem \ref{thmub2}, we can conclude that for any symplectic partition $\ul{p}>\ul{p}_1$, $\pi$ has no nonzero Fourier coefficients attached to the partition $\ul{p}$, in particular, $\ul{p} \notin \mathfrak{p}^m(\pi)$.


\textbf{Proof of Part (2).}\
By Definition \ref{def2},
\begin{align}\label{sec5equ12}
\begin{split}
& \eta_{\frak{so}_{2n+1}(\BC), \frak{sp}_{2n}(\BC)}([(\prod_{j=1}^t n_j^2) (\prod_{(\chi,m,\alpha) \in \textbf{e}} m^2) (\prod_{i=1}^l (2m_i+1))])\\
=  \ & (([(\prod_{j=1}^t n_j^2) (\prod_{(\chi,m,\alpha) \in \textbf{e}} m^2) (\prod_{i=1}^l (2m_i+1))]^-)_{\Sp})^t.
\end{split}
\end{align}
On the other hand, from the proof of Theorem 6.3.11 of \cite{CM93}, it is easy to see that
\begin{align}\label{sec5equ13}
\begin{split}
& ([(\prod_{j=1}^t n_j^2) (\prod_{(\chi,m,\alpha) \in \textbf{e}} m^2) (\prod_{i=2}^l (2m_i+1))_{\Sp}(2m_1)]^t)_{\Sp}\\
=  \ & ([(\prod_{j=1}^t n_j^2) (\prod_{(\chi,m,\alpha) \in \textbf{e}} m^2) (\prod_{i=2}^l (2m_i+1))_{\Sp}(2m_1)]^{\Sp})^t.
\end{split}
\end{align}
Comparing the right hand sides of \eqref{sec5equ12} and \eqref{sec5equ13}, we only need to show that
\begin{align}\label{sec5equ14}
\begin{split}
& ([(\prod_{j=1}^t n_j^2) (\prod_{(\chi,m,\alpha) \in \textbf{e}} m^2) (\prod_{i=1}^l (2m_i+1))]^-)_{\Sp}\\
=  \ & [(\prod_{j=1}^t n_j^2) (\prod_{(\chi,m,\alpha) \in \textbf{e}} m^2) (\prod_{i=2}^l (2m_i+1))_{\Sp}(2m_1)]^{\Sp}.
\end{split}
\end{align}

Let $J=\{1, 2, \ldots ,t\}$.
We rewrite the partition $$[(\prod_{j=1}^t n_j^2) (\prod_{(\chi,m,\alpha) \in \textbf{e}} m^2) (\prod_{i=1}^l (2m_i+1))]$$
as follows:
$$[(\prod_{j \in J_0}  n_j^2) (\prod_{(\chi,m,\alpha) \in \textbf{e}_0} m^2) (\prod_{i=1}^l (2m_i+1)^{f_i}(\prod_{j \in J_i}  n_j^2) (\prod_{(\chi,m,\alpha) \in \textbf{e}_i} m^2))],$$
such that $n_j, m \geq 2m_l+1$, for $j \in J_0$, $(\chi,m,\alpha) \in \textbf{e}_0$;
and $2m_{i-1} < n_j, m < 2m_i+1$, for $j \in J_i$, $(\chi,m,\alpha) \in \textbf{e}_i$, where we let $m_0=-1$;
and $f_i \geq 1$ odd, for $1 \leq i \leq l$.

As in the proof of Part (1), we still let $l=2s+1$.
Since $2m_1+1<2m_2+1<\cdots<2m_{2s+1}+1$, by the recipe for symplectic collapse
(see Theorem 6.3.8 of \cite{CM93}, also the beginning of Section 4),
\begin{align*}
& (\prod_{i=2}^l (2m_i+1))_{\Sp}\\
=  \ & [(2m_{2s+1})(2m_{2s}+2)\cdots(2m_5)(2m_4+2)(2m_3)(2m_2+2)].
\end{align*}

Then
\begin{align}\label{sec5equ15}
\begin{split}
& [(\prod_{j=1}^t n_j^2) (\prod_{(\chi,m,\alpha) \in \textbf{e}} m^2) (\prod_{i=2}^l (2m_i+1))_{\Sp}(2m_1)]^{\Sp}\\
=  \ & [(\prod_{j \in J_0}  n_j^2) (\prod_{(\chi,m,\alpha) \in \textbf{e}_0} m^2) \\
& \cdot (\prod_{i=1}^s
((2m_{2i+1}+1)^{f_{2i+1}-1}(2m_{2i+1})(\prod_{j \in J_{2i+1}}  n_j^2) (\prod_{(\chi,m,\alpha) \in \textbf{e}_{2i+1}} m^2)\\
& \cdot (2m_{2i}+2)(2m_{2i}+1)^{f_{2i}-1}(\prod_{j \in J_{2i}}  n_j^2) (\prod_{(\chi,m,\alpha) \in \textbf{e}_{2i}} m^2)))\\
& \cdot (2m_1+1)^{f_1-1} (2m_1) (\prod_{j \in J_i}  n_j^2) (\prod_{(\chi,m,\alpha) \in \textbf{e}_i} m^2)]^{\Sp}.
\end{split}
\end{align}

It is easy to see that during the operations of $[\,]^-$, $\Sp$-collapse and $\Sp$-expansion, the part of $[(\prod_{j \in J_0}  n_j^2) (\prod_{(\chi,m,\alpha) \in \textbf{e}_0} m^2)]$ will not change. Therefore, we only need to show that
\begin{align}\label{sec5equ16}
\begin{split}
& ([\prod_{i=1}^{2s+1} (2m_i+1)^{f_i}(\prod_{j \in J_i}  n_j^2) (\prod_{(\chi,m,\alpha) \in \textbf{e}_i} m^2)]^-)_{\Sp}\\
=  \ & [(\prod_{i=1}^s
((2m_{2i+1}+1)^{f_{2i+1}-1}(2m_{2i+1})(\prod_{j \in J_{2i+1}}  n_j^2) (\prod_{(\chi,m,\alpha) \in \textbf{e}_{2i+1}} m^2)\\
& \cdot (2m_{2i}+2)(2m_{2i}+1)^{f_{2i}-1}(\prod_{j \in J_{2i}}  n_j^2) (\prod_{(\chi,m,\alpha) \in \textbf{e}_{2i}} m^2)))\\
& \cdot (2m_1+1)^{f_1-1} (2m_1) (\prod_{j \in J_i}  n_j^2) (\prod_{(\chi,m,\alpha) \in \textbf{e}_i} m^2)]^{\Sp}.
\end{split}
\end{align}

For $1 \leq i \leq 2s+1$,
write the partition
$$[(2m_i+1)^{f_i}(\prod_{j \in J_i}  n_j^2) (\prod_{(\chi,m,\alpha) \in \textbf{e}_i} m^2)]$$
as $[(2m_i+1)^{f_i}p_{i,1}^{2g_{i,1}}\cdots p_{i,r_i}^{2g_{i,r_i}}]$,
with $2m_i+1 > 2m_i \geq p_{i,1} > \cdots >p_{i,r_i}$.

We need to consider two cases: Case (1),
$p_{1,r_1} = 2q_{1,r_1}+1$, odd; and Case (2), $p_{1,r_1} = 2q_{1,r_1}$, even.

For Case (1).
\begin{align}\label{sec5equ17}
\begin{split}
& ([\prod_{i=1}^{2s+1} (2m_i+1)^{f_i}(\prod_{j \in J_i}  n_j^2) (\prod_{(\chi,m,\alpha) \in \textbf{e}_i} m^2)]^-)_{\Sp}\\
=  \ & [\prod_{i=2}^{2s+1} (2m_i+1)^{f_i}p_{i,1}^{2g_{i,1}}\cdots p_{i,r_i}^{2g_{i,r_i}}\\
& \cdot (2m_1+1)^{f_1}p_{1,1}^{2g_{1,1}}\cdots
p_{1,r_1-1}^{2g_{1,r_1-1}} (2q_{1,r_1}+1)^{2g_{1,r_1}-1}(2q_{1,r_1})]_{\Sp}.
\end{split}
\end{align}
For $2 \leq i \leq 2s+1$, assume that all the odd parts in $\{p_{i,1}, \ldots, p_{i,r_i}\}$ are
$\{(2q_{i,1}+1), \ldots, (2q_{i,t_i}+1)\}$, with
$2q_{i,1}+1 > \cdots > 2q_{i,t_i}+1$.
And assume that all the odd parts in $\{p_{1,1}, \ldots, p_{1,r_1-1}\}$ are
$\{(2q_{1,1}+1), \ldots, (2q_{1,t_1}+1)\}$, with
$2q_{1,1}+1 > \cdots > 2q_{1,t_1}+1$.
For $1 \leq i \leq 2s+1$, and $1 \leq j \leq t_i$, we assume that the exponent of $2q_{i,j}+1$ is $h_{i,j}$.
Then by the recipe in Theorem 6.3.8 of \cite{CM93} (see the beginning of Section 4), to get the $\Sp$-collapse in the right hand side of \eqref{sec5equ17}, we just have to do the following:
\begin{itemize}
\item for $0 \leq i \leq s$,
replace $(2m_{2i+1}+1)^{f_{2i+1}}(2m_{2i}+1)^{f_{2i}}$
by $(2m_{2i+1}+1)^{f_{2i+1}-1}(2m_{2i+1})(2m_{2i}+2)(2m_{2i}+1)^{f_{2i}-1}$,
and for $1 \leq j \leq t_{2i+1}$, replace
$(2q_{{2i+1},j}+1)^{h_{{2i+1},j}}$ by
$$(2q_{{2i+1},j}+2))(2q_{{2i+1},j}+1)^{h_{{2i+1},j}-2}(2q_{{2i+1},j});$$
\item replace $(2m_1+1)^{f_1}$ by $(2m_1+1)^{f_1-1}(2m_1)$, and
replace $(2q_{1,r_1}+1)^{2g_{1,r_1}-1}$
by $(2q_{1,r_1}+2)(2q_{1,r_1}+1)^{2g_{1,r_1}-2}$.
\end{itemize}
On the other hand,
by the recipe in Theorem 6.3.9 of \cite{CM93} (see the beginning of Section 4), to get the $\Sp$-expansion in the right hand side of \eqref{sec5equ16}, we just have to do the following:
\begin{itemize}
\item for $0 \leq i \leq s$,
$1 \leq j \leq t_{2i+1}$, replace
$(2q_{{2i+1},j}+1)^{h_{{2i+1},j}}$ by
$$(2q_{{2i+1},j}+2))(2q_{{2i+1},j}+1)^{h_{{2i+1},j}-2}(2q_{{2i+1},j});$$
\item replace $(2q_{1,r_1}+1)^{2g_{1,r_1}}$
by $(2q_{1,r_1}+2)(2q_{1,r_1}+1)^{2g_{1,r_1}-2}(2q_{1,r_1})$.
\end{itemize}
Therefore, we have proved the equality in \eqref{sec5equ16} for Case (1).

For Case (2).
\begin{align}\label{sec5equ18}
\begin{split}
& ([\prod_{i=1}^{2s+1} (2m_i+1)^{f_i}(\prod_{j \in J_i}  n_j^2) (\prod_{(\chi,m,\alpha) \in \textbf{e}_i} m^2)]^-)_{\Sp}\\
=  \ & [\prod_{i=2}^{2s+1} (2m_i+1)^{f_i}p_{i,1}^{2g_{i,1}}\cdots p_{i,r_i}^{2g_{i,r_i}}\\
& \cdot (2m_1+1)^{f_1}p_{1,1}^{2g_{1,1}}\cdots
p_{1,r_1-1}^{2g_{1,r_1-1}} (2q_{1,r_1})^{2g_{1,r_1}-1}(2q_{1,r_1}-1)]_{\Sp}.
\end{split}
\end{align}
As in Case (1),
for $2 \leq i \leq 2s+1$, assume that all the odd parts in $\{p_{i,1}, \ldots, p_{i,r_i}\}$ are
$\{(2q_{i,1}+1), \ldots, (2q_{i,t_i}+1)\}$, with
$2q_{i,1}+1 > \cdots > 2q_{i,t_i}+1$.
And assume that all the odd parts in $\{p_{1,1}, \ldots, p_{1,r_1-1}\}$ are
$\{(2q_{1,1}+1), \ldots, (2q_{1,t_1}+1)\}$, with
$2q_{1,1}+1 > \cdots > 2q_{1,t_1}+1$.
For $1 \leq i \leq 2s+1$, and $1 \leq j \leq t_i$, we assume that the exponent of $2q_{i,j}+1$ is $h_{i,j}$.
Then by the recipe in Theorem 6.3.8 of \cite{CM93} (see the beginning of Section 4), to get the $\Sp$-collapse in the right hand side of \eqref{sec5equ18}, we just have to do the following:
\begin{itemize}
\item for $0 \leq i \leq s$,
replace $(2m_{2i+1}+1)^{f_{2i+1}}(2m_{2i}+1)^{f_{2i}}$
by $(2m_{2i+1}+1)^{f_{2i+1}-1}(2m_{2i+1})(2m_{2i}+2)(2m_{2i}+1)^{f_{2i}-1}$,
and for $1 \leq j \leq t_{2i+1}$, replace
$(2q_{{2i+1},j}+1)^{h_{{2i+1},j}}$ by
$$(2q_{{2i+1},j}+2))(2q_{{2i+1},j}+1)^{h_{{2i+1},j}-2}(2q_{{2i+1},j});$$
\item replace $(2m_1+1)^{f_1}$ by $(2m_1+1)^{f_1-1}(2m_1)$, and
replace $(2q_{1,r_1}-1)$
by $(2q_{1,r_1})$.
\end{itemize}
On the other hand,
by the recipe in Theorem 6.3.9 of \cite{CM93} (also see the beginning of Section 4), to get the $\Sp$-expansion in the right hand side of \eqref{sec5equ16}, we just have to do the following:
\begin{itemize}
\item for $0 \leq i \leq s$, and
$1 \leq j \leq t_{2i+1}$, replace
$(2q_{{2i+1},j}+1)^{h_{{2i+1},j}}$ by
$$(2q_{{2i+1},j}+2))(2q_{{2i+1},j}+1)^{h_{{2i+1},j}-2}(2q_{{2i+1},j}).$$
\end{itemize}
Therefore, we also have proved the equality in \eqref{sec5equ16} for Case (2).
Hence, we have proved that
$$\ul{p}_1 = \eta_{\frak{so}_{2n+1}(\BC), \frak{sp}_{2n}(\BC)}([(\prod_{j=1}^t n_j^2) (\prod_{(\chi,m,\alpha) \in \textbf{e}} m^2) (\prod_{i=1}^l (2m_i+1))]).$$
This finishes the proof of Part (2), and completes the proof of the theorem.
\end{proof}

The proof of Theorem \ref{thmub1} easily implies the following corollary.

\begin{cor}\label{cor1}
Let $\pi$ be an irreducible unitary automorphic representation of $\Sp_{2n}(\BA)$ which has an unramified component $\sigma$ of \textbf{Type III} as in \ref{typeIII}. Then,
for any symplectic partition $\ul{p}$ of $2n$ which is bigger than the partition $\ul{p}_1$ in Theorem \ref{thmub1} under the lexicographical ordering,
$\pi$ has no non-vanishing Fourier coefficients attached to $\ul{p}$, in particular, $\ul{p} \notin \mathfrak{p}^m(\pi)$.
\end{cor}

\begin{rmk}
Theorems \ref{thmub2}, \ref{thmub3} and \ref{thmub4} also have similar corollaries, if we replace the dominance ordering of partitions by the lexicographical ordering.
\end{rmk}

Applying Theorem \ref{thmub1}, we have the following analogue result, which is a generalization of Theorem \ref{thmub3}.

\begin{thm}\label{thmub4}
Let $\pi$ be an irreducible unitary automorphic representation of $\Sp_{2n}(\BA)$ which has, at one unramified local place $v$
an unramified component $\sigma_v$ of \textbf{Type IV} as in \ref{typeIV}. Then the following hold.
\begin{enumerate}
\item For any symplectic partition $\ul{p}$ of $2n$ with
$$\ul{p} > \ul{p}_1 := ([(\prod_{j=1}^t n_j^2) (\prod_{(\chi,m,\alpha) \in \textbf{e}} m^2) (\prod_{i=1}^k (2n_i+1))_{\Sp}]^t)_{\Sp},$$
$\pi$ has no non-vanishing Fourier coefficients attached to $\ul{p}$, in particular, $\ul{p} \notin \mathfrak{p}^m(\pi)$.
\item The partition $\ul{p}_1$ has the property that
$$\ul{p}_1 = \eta_{\frak{so}_{2n+1}(\BC), \frak{sp}_{2n}(\BC)}([(\prod_{j=1}^t n_j^2) (\prod_{(\chi,m,\alpha) \in \textbf{e}} m^2) (\prod_{i=1}^k (2n_i+1))(1)]).$$
\end{enumerate}
\end{thm}

\begin{proof}
By assumption, $\sigma_v$ corresponds to the following set of data:
\begin{align*}
& \{(\chi, m, \alpha) : (\chi, m, \alpha) \in \textbf{ e}\} \cup \{(\chi, n_i): 1 \leq i \leq t\}\\
\cup  \ & \{(\lambda_0, 2n_1+1), \ldots, (\lambda_0, 2n_k+1), (1_{\GL_1},1)\}.
\end{align*}
Rewrite $\sigma_v$ as $\lambda_0 \sigma'_v$, then it is easy to see that $\sigma'_v$ is an irreducible  unramified
unitary representation corresponds to the following set of data:
\begin{align*}
& \{(\lambda_0\chi, m, \alpha) : (\chi, m, \alpha) \in \textbf{ e}\} \cup \{(\lambda_0\chi, n_i): 1 \leq i \leq t\}\\
\cup  \ & \{(1_{\GL_1}, 2n_1+1), \ldots, (1_{\GL_1}, 2n_k+1), (1_{\GL_1},1)\}.
\end{align*}
Applying a similar argument in the proof of Theorem \ref{thmub1} to $\sigma'_v$, we can
see that for any symplectic partition $\ul{p}$ of $2n$ with
$$\ul{p} > \ul{p}_1 := ([(\prod_{j=1}^t n_j^2) (\prod_{(\chi,m,\alpha) \in \textbf{e}} m^2) (\prod_{i=1}^k (2n_i+1))_{\Sp}]^t)_{\Sp},$$
$\pi$ has no non-vanishing Fourier coefficients attached to $\ul{p}$, in particular, $\ul{p} \notin \mathfrak{p}^m(\pi)$,
and
$$\ul{p}_1 = \eta_{\frak{so}_{2n+1}(\BC), \frak{sp}_{2n}(\BC)}([(\prod_{j=1}^t n_j^2) (\prod_{(\chi,m,\alpha) \in \textbf{e}} m^2) (\prod_{i=1}^k (2n_i+1))(1)]).$$
This completes the proof of the theorem.
\end{proof}


\section{Proofs of the main results}

In this section, we first establish in Section 6.1 a refined structure about the irreducible unramified representation corresponding
to the unramified local Arthur parameter $\psi_v$ for infinitely many local places where $\psi_v$ are unramified. This result is crucial
in the proof of the main result of the paper (Theorem 1.3) given in Section 6.2. In Section 6.3, we prove a result which is in fact
related to Part (2) of Conjecture 1.2.

\subsection{On square classes}
\begin{prop}\label{propsq}
For any finitely many non-square elements $\alpha_i \notin F^*/(F^*)^2, 1 \leq i \leq t$,  there are infinitely many finite places $v$ such that $\alpha_i \in (F_v^*)^2$, for any $1 \leq i \leq t$.
\end{prop}

\begin{proof}
First, it is easy to find a sufficiently large set of finitely many places $S$ which contains all the archimedean places, and $s:=\#(S) >t$,
such that $\alpha_{i}$'s are all non-square $S$-units. Let
$$\mathfrak{u}_S=\{x \in F^* : \lvert x \rvert_v = 1, \forall v \notin S\}$$
be the set of $S$-units.
Then, by the Dirichlet Unit Theorem (Page 105 of \cite{Lang94}), $\mathfrak{u}_S$ is a direct product of the roots of unity $U_F$ of $F$ with a free abelian group of rank $s-1$. Since $-1 \in U_F$, $U_F / (U_F)^2 \cong \BZ / 2 \BZ$.
Therefore, $\mathfrak{u}_S / \mathfrak{u}_S^2 \cong (\BZ / 2 \BZ)^s$,
which can be viewed as an $s$-dimensional vector space over $\BZ/2\BZ$.

Assume that $\{\epsilon_1, \ldots, \epsilon_{q}\}$ is a maximal multiplicatively independent subset of $\{\alpha_{1}, \ldots, \alpha_{t}\}$. Extend $\{\epsilon_1, \ldots, \epsilon_{q}\}$ to a set of generators of $\mathfrak{u}_S / \mathfrak{u}_S^2$: $\{\epsilon_1, \ldots, \epsilon_{s}\}$. Note that $q \leq t < s$, and any product of distinct $\epsilon_i$'s is also not a square.
Let $K=F(\sqrt{\epsilon_1}, \ldots, \sqrt{\epsilon_{s-1}})$. Then it is clear that $\epsilon_s$ is not a square in $K$.
By the Global Square Theorem (Theorem 65:15 of \cite{O71}) which is a special case of the result on Page 194 of \cite{Lang94},
there are infinitely many places $\omega \in S'$ of $K$, such that $\epsilon_s \notin (K_{\omega}^*)^2$. These places induce infinitely many places which are not in $S$.

For any $\omega \in S'$, such that $\omega \vert v$ and $v \notin S$, then $\epsilon_s \notin (K_{\omega}^*)^2$, which implies that $\epsilon_s \notin (F_{v}^*)^2$.
For any $1 \leq i \leq s-1$, since $\epsilon_i \in (K_{\omega}^*)^2$, $\epsilon_s \epsilon_i \notin (K_{\omega}^*)^2$, and hence $\epsilon_s \epsilon_i \notin (F_{v}^*)^2$. For any $1 \leq i \leq s-1$, since both $\epsilon_s$ and $\epsilon_s \epsilon_i$
are non-square units in $\CO_v$, it is easy to see that they are in the same square class, which implies that $\epsilon_i \in F_v^2$.
Therefore, there are infinitely many finite places $v$ which are not in $S$, such that $\epsilon_i \in F_v^2$, for any $1 \leq i \leq s-1$.
Since $\alpha_{i}$'s are generated by $\{\epsilon_1, \ldots, \epsilon_{q}\}$, they are all squares in $F_v$ for these $v$.

This completes the proof of the proposition.
\end{proof}

\begin{rmk}\label{rmk2}
Applying Dirichlet's theorem on primes in arithmetic progressions, the law of quadratic reciprocity, and the Chinese remainder theorem, it is easy to see that for any $M > 0$, there are infinitely many primes $p$ such that the numbers $1, 2, \ldots, M$ are all residues modulo $p$.
Proposition \ref{propsq} generalizes this result to arbitrary number fields.
\end{rmk}

The following proposition gives more structure on the global non-square classes occurring in any global Arthur parameter,
which is of interest for future applications.

\begin{prop}\label{propap}
Given any $\psi = \boxplus_{i=1}^r (\tau_i, b_i) \in \wt{\Psi}_2(\Sp_{2n})$. Assume
that $\{\tau_{i_1}, \ldots, \tau_{i_q}\}$ is a multi-set of all the $\tau$'s with non-trivial central characters, and $\omega_{\tau_{i_j}}=\chi_{\alpha_{i_j}}$,
where $\alpha_{i_j} \in F^* / (F^*)^2$.
Let $S$ be any set of finitely many places which contains all the archimedean places, $s := \#(S)$,
such that $\alpha_{i_j}$'s are in $\mathfrak{u}_S$-the set of all $S$-units.
Also assume that $\{\epsilon_1, \ldots, \epsilon_{s}\}$ is any set of generators of $\mathfrak{u}_S / \mathfrak{u}_S^2$,
and $\alpha_{i_j} = \epsilon_1^{v_{1,j}}
\cdots \epsilon_s^{v_{s,j}} \delta^2$,
where $v_{l,j}=0$ or $1$, $\delta \in \mathfrak{u}_S$. Then
$\sum_{j=1}^q v_{l,j}$ must be even, for any $1 \leq l \leq s$.
\end{prop}

\begin{proof}
Assume on the contrary that there is an $1 \leq l \leq s$, such that $\sum_{j=1}^q v_{l,j}$ is odd.
By a similar argument as in the proof of Proposition \ref{propsq},
there are infinitely many places $v$ which are not in $S$,
such that $\epsilon_l \notin F_v^2$,
$\epsilon_k \in F_v^2$, for any $1 \leq k \neq l \leq s$.
Then, it is easy to see that
$\prod_{i=1}^r \omega_{\tau_{i,v}}^{b_i} \neq 1$ for these $v$, as central characters of $\GL_{2n+1}(F_v)$, applying the multiplicativity of local Hilbert symbol. Therefore,
$\prod_{i=1}^r \omega_{\tau_{i}}^{b_i} \neq 1$, as central characters of $\GL_{2n+1}(\BA)$, which contradicts the definition of Arthur parameters.
Therefore, $\sum_{j=1}^q v_{l,j}$ must be even, for any $1 \leq l \leq s$.
\end{proof}

\subsection{The completion of the proof of Theorem \ref{main}}

Given any $\psi = \boxplus_{i=1}^r (\tau_i, b_i) \in \wt{\Psi}_2(\Sp_{2n})$. Assume that $\{\tau_{i_1}, \ldots, \tau_{i_q}\}$ is a multi-set of all the $\tau$'s with non-trivial central characters.
Since all $\tau_{i_j}$'s are self-dual, the central characters $\omega_{\tau_{i_j}}$'s are all quadratic characters, which are parametrized by global non-square elements. Assume that $\omega_{\tau_{i_j}}=\chi_{\alpha_{i_j}}$,
where $\alpha_{i_j} \in F^* / (F^*)^2$,
and $\chi_{\alpha_{i_j}}$ is the quadratic character given by the global Hilbert symbol $(\cdot, \alpha_{i_j})$.
Note that $\{\alpha_{i_1}, \ldots, \alpha_{i_q}\}$ is a multi-set.

By Proposition \ref{propsq}, there are infinitely many finite places $v$, such that $\alpha_{i_j}$'s are all squares in
$F_v$, that is, $\omega_{\tau_{i_j,v}}$'s are all trivial.
Therefore, for the given $\psi$,
there are infinitely many finite places $v$ such that all $\tau_{i,v}$'s have trivial central characters. From
the discussion in Section 2.2, for any $\pi \in \wt{\Pi}_{\psi}(\epsilon_{\psi})$, there is a finite local place $v$ with such a property
that $\pi_v$ is an irreducible unramified unitary representation of \textbf{Type III} as in \ref{typeIII}.

Indeed, the difference between irreducible unramified unitary representations of \textbf{Type III} as in \ref{typeIII} and general irreducible unramified unitary representations is the strongly negative part $\sigma_{sn}$. In general, via classification, $\sigma_{sn}$ involves two kinds of Jordan blocks $(1_{\GL_1}, 2m_i+1)$ and $(\lambda_0, 2n_i+1)$. For a general irreducible unramified unitary representation, in order to be of \textbf{Type III}, $\sigma_{sn}$ should only involve Jordan blocks $(1_{\GL_1}, 2m_i+1)$. From the discussion in Section 2.2, for a finite place $v$ such that all $\pi_{i,v}$'s have trivial central characters, all Jordan blocks involved in $\pi_v$ either have even multiplicities which will not occur in the strongly negative part, or are only $(1_{\GL_1}, 2m_i+1)$'s with odd multiplicities which will occur in the strongly negative part, hence $\pi_v$ is of \textbf{Type III}.

By Theorem \ref{thmub1}, for any symplectic partition $\ul{p}$ of $2n$ with $\ul{p} > \eta_{{\frak{g}^\vee,\frak{g}}}(\underline{p}(\psi))$,
$\pi$ has no non-vanishing Fourier coefficients attached to $\ul{p}$,
in particular, $\ul{p} \notin \mathfrak{p}^m(\pi)$.
This completes the proof of Theorem \ref{main}.

\subsection{About Part (2) of Conjecture \ref{cubmfc}}

Part (2) of Conjecture \ref{cubmfc} can be rephrased as follows: given
any $\psi\in\wt{\Psi}_2(\Sp_{2n})$ and any $\pi\in\wt{\Pi}_{\psi}(\epsilon_\psi)$,
 for any symplectic partition $\ul{p}$ which is not related to $\eta_{{\frak{g}^\vee,\frak{g}}}(\ul{p}(\psi))$ under the usual ordering of partitions, that is, the dominance ordering, $\pi$ has no non-vanishing Fourier coefficients attached to $\ul{p}$, in particular, $\ul{p}\notin\frak{p}^m(\pi)$.

Assume that $\ul{p}$ is a symplectic partition which is not related to $\eta_{{\frak{g}^\vee,\frak{g}}}(\ul{p}(\psi))$ under the dominance ordering.
If we consider the lexicographical ordering of partitions, which is a total ordering, then in general there are two cases:
\begin{enumerate}
\item[(2-1)] $\ul{p}$ is bigger than $\eta_{{\frak{g}^\vee,\frak{g}}}(\ul{p}(\psi))$ under the lexicographical ordering;
\item[(2-2)] $\ul{p}$ is smaller than $\eta_{{\frak{g}^\vee,\frak{g}}}(\ul{p}(\psi))$ under the lexicographical ordering.
\end{enumerate}

Replacing Theorem \ref{thmub1} by Corollary \ref{cor1} in the proof of Theorem \ref{main}, we can easily get the following result towards confirming Part (2) of Conjecture \ref{cubmfc}.

\begin{prop}\label{prop3}
Given any $\psi\in\wt{\Psi}_2(\Sp_{2n})$ and any $\pi\in\wt{\Pi}_{\psi}(\epsilon_\psi)$.
Assume that $\ul{p}$ is a symplectic partition which is not related to $\eta_{{\frak{g}^\vee,\frak{g}}}(\ul{p}(\psi))$ under the dominance ordering.
If $\ul{p}$ is bigger than $\eta_{{\frak{g}^\vee,\frak{g}}}(\ul{p}(\psi))$ under the lexicographical ordering, then $\pi$ has no non-vanishing Fourier coefficients attached to $\ul{p}$, in particular, $\ul{p}\notin\frak{p}^m(\pi)$.
\end{prop}

\begin{rmk}\label{rmk8}
For certain global Arthur parameters of symplectic groups, if $\ul{p}$ is a symplectic partition which is not related to $\eta_{{\frak{g}^\vee,\frak{g}}}(\ul{p}(\psi))$ under the dominance ordering, then $\ul{p}$ is automatically bigger than $\eta_{{\frak{g}^\vee,\frak{g}}}(\ul{p}(\psi))$ under the lexicographical ordering. For example, the global Arthur parameters for $\Sp_{4mn}$ considered in \cite{Liu13}: $\psi=(\tau, 2m) \boxplus (1_{\GL_1(\BA)}, 1)$, where $\tau$ is an irreducible cuspidal automorphic representation of
$\GL_{2n}(\BA)$, with the properties that
$L(s, \tau, \wedge^2)$ has a simple
pole at $s=1$, and
$L(\frac{1}{2}, \tau) \neq 0$. By definition, $\ul{p}(\psi)=[(2m)^{2n}1]$. By Definition \ref{def2},
$\eta_{{\frak{g}^\vee,\frak{g}}}(\ul{p}(\psi))=[(2n)^{2m}]$. Then it is easy to see that if a symplectic partition $\ul{p}$ is not related to $[(2n)^{2m}]$ under the dominance ordering, then it is automatically bigger than $[(2n)^{2m}]$ under the lexicographical ordering.
\end{rmk}


\end{document}